\documentclass[11pt]{amsart}

\usepackage{hyperref}
\hypersetup{
	colorlinks=true,     
	linkcolor= blue,     
	citecolor=blue,        
	filecolor=magenta,     
	urlcolor=cyan           
}

\usepackage{geometry}
\usepackage{amsmath,amsfonts,amssymb,amsthm}
\usepackage{bm}
\usepackage{bbm}
\usepackage[all]{xy}
\usepackage{mathrsfs}
\usepackage{stmaryrd}
\usepackage[OT2,T1]{fontenc}
\usepackage[latin1]{inputenc}

\DeclareSymbolFont{cyrletters}{OT2}{wncyr}{m}{n}
\DeclareMathSymbol{\Sha}{\mathalpha}{cyrletters}{"58}

\DeclareSymbolFont{cyrletters}{OT2}{wncyr}{m}{n}
\DeclareMathSymbol{\Sha}{\mathalpha}{cyrletters}{"58}

\def \ooverline #1#2#3%
{\mkern #1mu \overline{\mkern -#1mu #3 \mkern -#2mu }\mkern #2mu }

\DeclareMathOperator{\Sel}{Sel}

\DeclareMathOperator{\Gal}{Gal}

\DeclareMathOperator{\Ad}{Ad}

\DeclareMathOperator{\Def}{Def}

\DeclareMathOperator{\End}{\mathsf{End}} 
\DeclareMathOperator{\Hom}{\mathsf{Hom}}
\DeclareMathOperator{\sHom}{\mathsf{sHom}} 

\DeclareMathOperator{\Ker}{Ker} 
\DeclareMathOperator{\im}{Im} 
\DeclareMathOperator{\Coker}{Coker}

\renewcommand{\H}{\mathrm{H}}

\DeclareMathOperator{\Spec}{Spec}

\DeclareMathOperator{\ord}{n.o}
\DeclareMathOperator{\pr}{pr}

\DeclareMathOperator{\Lie}{Lie}
 
\DeclareMathOperator{\GL}{\mathsf{GL}}

\DeclareMathOperator{\diag}{diag}

\newcommand{\cC}{\mathcal{C}}
\newcommand{\cD}{\mathcal{D}}

\newcommand{\cF}{\mathcal{F}}

\newcommand{\cM}{\mathcal{M}}

\newcommand{\cO}{\mathcal{O}}

\newcommand{\cR}{\mathcal{R}}

\def\im{\mathop{\rm im}}
\def\ker{\mathop{\rm Ker}}

\def\g{{\mathfrak g}}
\def\t{{\mathfrak t}}

\def\z{{\mathfrak z}}
\def\b{{\mathfrak b}}
\def\n{{\mathfrak n}}

\def\m{{\mathfrak m}}

\def\N{{\mathbb N}}

\def\Q{{\mathbb Q}}
\def\e{{\bf e}}

\def\Z{{\mathbb Z}}

\def\adots{\mathinner{\mkern2mu\raise1pt\hbox{.}\mkern3mu\raise4pt\hbox{.}\mkern1mu\raise7pt\hbox{.}}}

\theoremstyle{plain}
\newtheorem{thm}{Theorem}[section]
\newtheorem{pro}[thm]{Proposition}
\newtheorem{lem}[thm]{Lemma}

\newtheorem{cor}[thm]{Corollary}

\theoremstyle{definition}
\newtheorem{de}[thm]{Definition}
\newtheorem{expl}[thm]{Example}

\theoremstyle{remark}
\newtheorem{rem}[thm]{Remark}

\newcommand{\thmref}[1]{Theorem~\ref{#1}}
\newcommand{\propref}[1]{Proposition~\ref{#1}}
\newcommand{\lemref}[1]{Lemma~\ref{#1}}

\newcommand{\defref}[1]{Definition~\ref{#1}}

\newcommand{\remref}[1]{Remark~\ref{#1}}

\newcommand{\RNum}[1]{\uppercase\expandafter{\romannumeral #1\relax}}

\def\O{{\mathcal{O}}}
\def\gt{{\mathfrak t}}

\newcommand{\Der}{{\mathrm {Der}}}
\newcommand\Map{{\mathrm {Map}}}

\newcommand{\dslash}{{/\mkern-6mu/}}

\newcommand\FFS{{\mathbf {FFS}}}
\newcommand\FS{{\mathbf {FS}}}

\newcommand\PsCh{{\mathrm {PsCh}}}

\newcommand\aDef{{\mathrm {aDef}}}
\newcommand\bDef{{\mathrm {bDef}}}

\newcommand\ad{{\mathrm {ad}}}
\newcommand\op{{\mathrm {op}}}

\newcommand\Cat{{\mathbf {Cat}}}

\newcommand\Gp{{\mathbf {Gp}}}
\newcommand\Gpd{{\mathbf {Gpd}}}
\newcommand\semGp{{\mathbf {semGp}}}
\newcommand\Set{{\mathbf {Sets}}}
\newcommand\sSet{{s\mathbf {Sets}}}

\newcommand\Alg{{\mathbf {Alg}}}

\newcommand\CR{{\mathbf {CR}}}
\newcommand\sCR{{s\mathbf {CR}}}
\newcommand\Art{{ \mathbf {Art} }}

\newcommand\sArtT{{{}_{\O} \backslash s\mathbf {Art} /_{T}}}
\newcommand\sArt{{{}_{\O} \backslash s\mathbf {Art} /_{k}}}
\newcommand\Mod{{\mathbf {Mod}}}
\newcommand\sMod{{s\mathbf {Mod}}}
\newcommand{\Ch}{{\mathbf {Ch}}}
\newcommand\hocolim{{\mathrm{hocolim}}}
\newcommand\holim{{\mathrm{holim}}}
\newcommand{\Ho}{{\mathrm {Ho}}}
\newcommand{\DK}{{\mathrm {DK}}}
\newcommand{\Ex}{{\mathrm {Ex}}}
\newcommand{\hofib}{{\mathrm {hofib}}}

\newcommand{\Bi}{{\mathrm {Bi}}}
\newcommand{\BG}{{\mathcal {B}G}}
\newcommand{\BB}{{\mathcal {B}B}}
\newcommand{\BZ}{{\mathcal {B}Z}}
\newcommand{\barBG}{{\bar{\mathcal {B}}G}}

\newcommand\loc{{\mathrm {loc}}}

\title{Simplicial Galois Deformation Functors}
\author{Y. Cai and J. Tilouine}
\thanks{ This research was supported in part by the International Centre for Theoretical Sciences (ICTS) 
during a visit for participating in the program Perfectoid spaces (Code: ICTS/perfectoid2019/09).
 The authors are partially supported by the ANR grant CoLoSS ANR-19-PRC.} 
\date{}
\begin{document}

	\begin{abstract}
		
	In \cite{GV18}, the authors showed the importance of studying simplicial generalizations of Galois deformation functors. 
	They established a precise link between the simplicial
	universal deformation ring $R$ prorepresenting such a deformation problem (with local conditions) and a derived Hecke algebra. 
	Here we focus on the algebraic part of 
	their study which we complete in two directions. First, we introduce the notion of simplicial pseudo-characters and prove relations between
	the (derived) deformation functors of simplicial pseudo-characters and that of simplicial Galois representations.
	Secondly, we define the relative cotangent complex of a simplicial deformation functor and, in the ordinary case, we relate it 
	to the relative complex of ordinary Galois cochains. Finally, we recall how the latter can be used to relate the fundamental group of $R$ to the ordinary
	dual adjoint Selmer group, by a homomorphism already introduced in \cite{GV18} and studied in greater generality in \cite{TU20}. 
		
	\end{abstract}

	\maketitle

	\tableofcontents

	\section{Introduction}
	
	Let $p$ be an odd prime.  Let $K$ be a $p$-adic field, let $\cO$ be its valuation ring, $\varpi$ be a uniformizing parameter and $k=\cO/(\varpi)$ be the residue field. Let $\Gamma$ be a profinite group satisfying 
	
	$(\Phi_p)$ the $p$-Frattini quotient $\Gamma/\Gamma^p(\Gamma,\Gamma)$ is finite. 
	
	For instance, $\Gamma$ could be $\Gal(F_S/F)$, the Galois group of the maximal $S$-ramified extension of a number field $F$ with $S$ finite.  Let $G$ be a split connected reductive group scheme over $\cO$.  Let $\overline{\rho}\colon \Gamma\to G(k)$ be a continuous Galois representation.  Assume it is absolutely $G$-irreducible, which means its image is not contained in $P(k)$ for a proper parabolic subgroup $P$ of $G$.  The goal of this paper is to present and develop some aspects of the fundamental work \cite{GV18} and the subsequent papers \cite{TU20} and \cite{Cai20}, by putting emphasis on the algebraic notion of simplicial deformation over simplicial Artin local $\cO$-algebras of $\overline{\rho}$.
	
	In the papers mentioned above, it is assumed that the given residual Galois representation is automorphic: $\overline{\rho}=\overline{\rho}_\pi$ for a cohomological cuspidal automorphic representatiton on the dual group of $G$ over a number field $F$; then the (classical and simplicial) deformation problems considered impose certain local deformation conditions satisfied by $\overline{\rho}$ at primes above $p$ and at ramification primes for $\pi$.  The fundamental insight of \cite{GV18} is to relate the corresponding universal simplicial deformation ring to a derived version of the Hecke algebra acting on the graded cohomology of a locally symmetric space.  Actually, the main result \cite[Theorem 14.1]{GV18} (slightly generalized in \cite{Cai20}) is that after localization at the non Eisenstein maximal ideal $\m$ of the Hecke algebra corresponding to $\overline{\rho}$, the integral graded cohomology in which $\pi$ occurs is free over the graded homotopy ring of the universal simplicial deformation ring (and the degree zero part of this ring is isomorphic to the top degree integral Hecke algebra).  This is therefore a result of automorphic nature.  
	
	Here, on the other hand, we want to focus on the purely algebraic machinery of simplicial deformations and pseudo-deformations and their (co)tangent complex for a general profinite group $\Gamma$ satisfying $(\Phi_p)$.  
	
	In \cite[Section 11]{Laf18}, V. Lafforgue introduced the notion of a pseudo-character for a split connected reductive group $G$.  He proved that this notion coincides with that of $G$-conjugacy classes of $G$-valued Galois representations over an algebraically closed field $E$.  The main ingredient of his proof is a criterion of semisimplicity for elements in $G(E)^n$ in terms of closed conjugacy class; it is due to Richardson in  characteristic zero. It has been generalized to the case of an algebraically closed field of arbitrary characteristic by \cite{BMR05} replacing semisimplicity by $G$-complete reducibility (see also \cite{Ser05} and \cite[Theorem 3.4]{BHKT19}). Note that absolute $G$-irreducibility implies $G$-complete reducibility.
	
	Using this (and a variant for Artin rings), Boeckle-Khare-Harris-Thorne \cite[Theorem 4.10]{BHKT19} proved a generalization of Carayol's result for any split reductive group $G$: any pseudo-deformation over $G$ of an absolutely $G$-irreducible representation $\overline{\rho}$ is a $G$-deformation.
	
	In section \ref{artpseudochar}, we reformulate the theory of \cite[Section 4]{BHKT19} in the language of simplicial deformation.  Our main results are \thmref{aDef} and \thmref{bDef}.  In Section \ref{derpseudochar}, we propose a generalization of this theory for derived deformations.  Unfortunately, the result in this context is only partial, but still instructive.
	
	In Sections \ref{cotangent}, after recalling the definition of the tangent and cotangent complexes and its calculation for a Galois deformation functor, we
introduce a relative version of the cotangent complex. In order to relate the cotangent complex of the universal simplicial ring $\cR$ prorepresenting a deformation functor to a Selmer group, we shall take $\Gamma=G_{F,S}$ for a number field $F$ and for $S$ equal to the set of places above $p$ and $\infty$, and we shall deal with the simplest sort of local conditions, namely unramified outside $p$ and ordinary at each place above $p$. We show that the cotangent complex $L_{\cR/\cO}\otimes_{\cR} T$ is related to the ordinary Galois cochain complex.  Note that here the base $T$ is arbitrary, whereas in \cite{GV18} and \cite{Cai20} it was mostly the case $T=k$.

Finally, in Section \ref{GVmorphism}, we recall how this is used
to define a homomorphism, first constructed in \cite[Lemma 15.1]{GV18} and generalized and studied in \cite{TU20}, which relates the fundamental group of the simplicial ordinary universal deformation ring and the ordinary dual adjoint Selmer group.
	
	This work started during the conference on $p$-adic automorphic forms and Perfectoids held in Bangalore in September 2019.  The authors greatly appreciated the excellent working atmosphere during their stay.

	\section{Classical and simplicial Galois deformation functors}

	\subsection{Classical deformations}
	
	Let $\Gamma$ be a profinite group which satisfies $(\Phi_p)$. When necessary, we view $\Gamma$ as projective limit of finite groups $\Gamma_i$. Let $\Art_{\cO}$ be the category of Artinian local $\cO$-algebras with residue field $k$.  Recall that the framed deformation functor $\cD^{\square}\colon \Art_{\cO} \to \Set$ of $\overline{\rho}$ is defined by associating $A\in\Art_{\cO}$ to the set of continuous liftings $\rho\colon \Gamma\to G(A)$ which make the following diagram commute:
	\begin{displaymath}
	\xymatrix{
		\Gamma \ar[r]^{ \rho }\ar[dr]^{\bar{\rho}} & G(A) \ar[d]\\
		& G(k)
	}
	\end{displaymath}
	Let $Z$ be the center of $G$ over $\cO$. We assume throughout it is a smooth group scheme over $\cO$.  Let $\widehat{G}(A)=\ker(G(A) \to G(k))$, resp. $\widehat{Z}(A)=\ker(Z(A) \to Z(k))$.  Let $\g=\Lie(G/\cO)$, resp. $\z=Lie(Z/\O)$ be the $\cO$-Lie algebra of $G$, resp. $Z$, and let $\g_k=\g\otimes_{\cO} k$, resp. $\z_k=\z\otimes_{\cO} k$.  The universal deformation functor $\cD=\Def_{\overline{\rho}}\colon \Art_{\O} \to \Set$ is defined by associating $A\in\Art_{\O}$ to the set of $\widehat{G}(A)$-conjugacy classes of $\cD^{\square}(A)$.  As an application of Schlessinger's criterion (see \cite[Theorem 2.11]{Sch68}), the functor $\cD^{\square}$ is pro-representable, and when $\bar{\rho}$ satisfies $H^{0}(\Gamma, \g_k) = \z_k$, the functor $\cD$ is pro-representable (see \cite[Theorem 3.3]{Til96}). 
	
	We shall consider (nearly) ordinary deformations.  In this case, we always suppose $\Gamma=G_{F,S}$, where $F$ is a number field and $S=S_p\cup S_\infty$ is the set of places above $p$ and $\infty$.  Note that $\Gamma$ is profinite and satisfies $(\Phi_p)$. For any $v\in S_p$, let $\Gamma_v=\Gal(\overline{F}_v/F_v)$.  Let $B=TN\subset G$ be a Borel subgroup scheme ($T$ is a maximal split torus and $N$ is the unipotent radical of $B$); all these groups are defined over $\cO$.  Let $\Phi$ be the root system associated to $(G,T)$ and $\Phi^+$ the subset of positive roots associated to $(G,B,T)$.  Assume that for any place $v\in S_p$, we have 
	
	$(Ord_v)$ there exists $\overline{g}_v\in G(k)$ such that $\overline{\rho}\vert_{\Gamma_v}$ takes values in $\overline{g}_v^{-1}\cdot B(k)\cdot\overline{g}_v$. 
	
	Let $\overline{\chi}_v\colon \Gamma_v\to T(k)$ be the reduction modulo $N(k)$ of $\overline{g}_v\cdot \overline{\rho}\vert_{\Gamma_v}\cdot\overline{g}_v^{-1}.$  Let $\omega\colon \Gamma_v\to k^\times$ be the mod. $p$ cyclotomic character.  We shall need the following conditions for $v\in S_p$:
	
	$(Reg_v)$ for any $\alpha\in\Phi^+$, $\alpha\circ\overline{\chi}_v\neq 1$,  and
	
	$(Reg^\ast_v)$ for any $\alpha\in\Phi^+$, $\alpha\circ\overline{\chi}_v\neq \omega$. 
	
	We can define the subfunctor $\cD^{\square,\ord}\subset \cD^\square$ of nearly ordinary liftings by the condition that $\rho\in \cD^{\square,\ord}$ if and only if for any place $v\in S_p$ there exists $g_v\in G(A)$  which lifts $\bar{g}_v$ such that $\rho\vert_{\Gamma_v}$ takes values in $g_v^{-1}\cdot B(k)\cdot g_v$. 
	Note that this implies that the homomorphism $\chi_{\rho,v}\colon\Gamma_v\to T(A)$ given by $g_v\cdot \rho\vert_{\Gamma_v}\cdot g_v^{-1}$ lifts $\overline{\chi}_v$.

	Similarly, we define the subfunctor $\cD^{\ord}\subset\cD$ of nearly ordinary deformations by $\cD^{\ord}(A)=\cD^{\square,\ord}(A)/\widehat{G}(A)$.
	
	Recall \cite[Proposition 6.2]{Til96}:
	
	\begin{pro} 
		Assume that $\H^0(\Gamma,\g_k)=\z_k$ and that $(Ord_v)$ and $(Reg_v)$ hold for all places $v\in S_p$.  
Then $\cD^{\ord}$ (and $\cD^{\square,\ord}$) is pro-representable, say by the complete noetherian local $\cO$-algebra $R^{\ord}$.   
	\end{pro}
	
	Note that the condition $(Reg^\ast_v)$ will occur later in the study of the cotangent complex in terms of the (nearly) ordinary Selmer complex.
	As noted in \cite[Chapter 8]{Til96}, the morphism of functors $\cD^{\ord}\to\prod_{v\in S_p} \Def_{\overline{\chi}_v}$ given by 
$[\rho]\mapsto (\chi_{\rho,v})_{v\in S_p}$ provides a structure of $\Lambda$-algebra on $R^{\ord}$ for an Iwasawa algebra $\Lambda$
called the Hida-Iwasawa algebra.

\begin{rem} A lifting $\rho\colon \Gamma\to G(A)$ of $\bar{\rho}$ is called ordinary of weight $\mu$ if for any $v\in S_p$, after conjugation by $g_v$, the cocharacter 
$\rho\vert_{Iv}\colon I_v\to T(A)=B(A)/N(A)$ is given (via the Artin reciprocity map $rec_v$) by
$\mu\circ rec_v^{-1}\colon I_v\to \cO_{F_v}^\times\to T(A)$. 

If we assume that $\bar{\rho}$ admits a lifting $\rho_0\colon\Gamma\to G(\cO)$ 
which is ordinary of weight $\mu$, we
can also consider the weight $\mu$ ordinary deformation problem, defined as the subfunctor $\cD^{\ord,\mu}\subset \cD^{\ord}$ where we impose the extra condition to $[\rho]$
that for any $v\in S_p$, after conjugation by some $g_v$, $\rho\vert_{Iv}\colon I_v\to T(A)=B(A)/N(A)$ is given (via the Artin reciprocity map $rec_v$) by
$\mu\circ rec_v^{-1}\colon I_v\to \cO_{F_v}^\times\to T(\cO)\to T(A)$. 
This problem is prorerepresentable as well, say by $R^{\ord}_\mu$. The difference is that $R^{\ord}$ has a natural structure of algebra over an Iwasawa algebra, 
while, if $\rho_0$ is automorphic, $R^{\ord}_\mu$ is often proven 
to be a finite $\cO$-algebra (see \cite{Wi95} or \cite{Ge19} for instance).
\end{rem}
	
	These functors have natural simplicial interpretations.

	\subsection{Simplicial reformulation of classical deformations}\label{subsectReform}
	
	In this section, we'll try to introduce the basic notions of simplicial homotopy theory and proceed at the same time to give a simplicial definition of the deformation functor of $\bar{\rho}$.

	Recall that a groupoid is a category such that all homomorphisms between two objects are isomorphisms.  Let $\Gpd$ be the category of small groupoids.  We have a functor $\Gp\to \Gpd$ from the category $\Gp$ of groups to $\Gpd$ sending a group $G$ to the groupoid with one object $\bullet$ and such that $\End(\bullet)=G$.
   
	A model category is a category with three classes of morphisms called weak equivalences, cofibrations and fibrations, satisfying five axioms, see \cite[Definition 7.1.3]{Hir03}.  The category of groups is not a model category.  But it is known (see \cite[Theorem 6.7]{Str00}) that the category of groupoids $\Gpd$ is a model category, where a morphism $f\colon G\to H$ is 
	\begin{enumerate}
		\item a weak equivalence if it is an equivalence of categories;
		\item a cofibration if it is injective on objects;
		\item a fibration if for all $a\in G$, $b\in H$ and $h\colon f(a) \to b$ there exists $g\colon a\to a^{\prime}$ such that $f(a^{\prime}) = b$ and $f(g)=h$.
	\end{enumerate}
	If $\cC$ is a model category, its homotopy category $\Ho(\cC)$ is the localization of $\cC$ at weak equivalences.  It comes with a functor $\cC\to \Ho(\cC)$ universal for the property of sending weak equivalences to isomorphisms.
  
	In $\Gpd$, the empty groupoid is the initial object and the unit groupoid consisting in a unique object with a unique isomorphism is the final object.  In a model category, a fibration, resp. cofibration, over the final object, resp. from the initial object, is called a fibrant, resp. cofibrant object.  Note that every object of $\Gpd$ is both cofibrant and fibrant, and the homotopy category $\Ho(\Gpd)$ is the quotient category of $\Gpd$ modding out natural isomorphisms.  If we regard a group $G$ as a one point groupoid, the functor $\Gp \to \Ho(\Gpd)$ so obtained has the effect of moding out conjugations, so, for any finite group $\Gamma_i$, we have 
	\begin{displaymath}
	\Hom_{\Gp}(\Gamma_{i}, G(A))/G^{ad}(A) \cong \Hom_{\Ho(\Gpd)}(\Gamma_{i}, G(A)).
	\end{displaymath}
	
	To construct the deformation functor, we first need to recall the construction of the classifying simplicial set $BG$ associated to a groupoid $G$.
	
	Let $\bm{\Delta}$ be the category whose objects are sets $[n]=\{0,\ldots,n\}$ and morphisms are non-decreasing maps. It is called the cosimplicial indexing category (see \cite[Definition 15.1.8]{Hir03}).  Given a category $\cC$, the category $s\cC$ of simplicial objects of $\cC$ is the category of contravariant functors $F\colon \bm{\Delta}\to\cC$.  In particular, $\sSet$ is the category of simplicial sets.  For any $n\geq 0$, let $\bm{\Delta}[n]$ be the simplicial set $$[k]\mapsto \Hom_{\bm{\Delta}}([k],[n]).$$  Note that the category $\sSet$ admits enriched homomorphisms: if $X,Y$ are two simplicial sets, there is a natural simplicial set $\sHom(X,Y)$ whose degree zero term is $\Hom_{\sSet}(X,Y)$. Actually, $$\sHom(X,Y)_n=\Hom_{\sSet}(X\times\bm{\Delta}[n],Y).$$
	
	For $X\in\sSet$, the morphism $(d_1,d_0)\colon X_1\to X_0\times X_0$ generates an equivalence relation $\widetilde{X}_1$. The zeroth homotopy set $\pi_0X$ is defined as the quotient set $X_0/\widetilde{X}_1$.  Let $X$ be fibrant and let $x\in X_0$; one can define for $i\geq 1$, the $i$-th homotopy set $\pi_i(X,x)$ as the quotient of the set of pointed morphisms $\Hom_{\sSet_\ast}(\bm{\Delta}[n],X)$ (morphisms sending the boundary $\partial\bm{\Delta}[n]$ to $x$) by the homotopy relation (see \cite[Section 8.3]{Weib94}).  Then $\pi_i(X,x)$ is naturally a group which is abelian when $i\geq2$ (see \cite[Theorem \RNum{1}.7.2]{GJ09}). 
	
	For $X\in \sSet$, let $\bm{\Delta}  X$ be the category whose objects are pairs $(n,\sigma)$ where $n\geq 0$ and $\sigma\colon\bm{\Delta}[n]\to  X$ is a morphism of simplicial sets, and morphisms $(n,\sigma)\to (m,\tau)$ are given by a non-decreasing map $\varphi\colon[n]\to[m]$ such that $\sigma=\tau\circ\varphi$.  The category $\bm{\Delta}  X$ is called the category of simplices of $X$ (see \cite[Definition 15.1.16]{Hir03}). 
	
	The following lemma is well-known:
	
	\begin{lem}
		Suppose $\cC$ is a category admitting colimits; let $F\colon \bm{\Delta} \to \cC$ be a covariant functor.  Let $F_{\ast}\colon \cC \to \sSet$ be the functor which sends $A\in \cC$ to the simplicial set $X=(X_n)_{n\geq 0}$ given by $X_n=\Hom_{\cC}(F([n]), A)$ at $n$-th simplicial degree, and let $F^{\ast}\colon \sSet \to \cC$ be the functor which sends $X \in \sSet$ to $\varinjlim\limits_{(n,\sigma) \in \bm{\Delta}  X} F(\sigma)$.  Then $F^{\ast}$ is left adjoint to $F_{\ast}$.
	\end{lem}
	
	\begin{proof}
		It's clear that $F_{\ast}$ is well-defined, and $F^{\ast}$ is well-defined since every simplicial set morphism $f \colon X \to Y$ induces a functor $\bm{\Delta}  X \to\bm{\Delta}  Y$.  For $X\in \sSet$ and $A\in \cC$, we have
		\begin{align*}
		\Hom_{\cC}(F^{\ast}(X),A) &\cong \varprojlim\limits_{(\Delta[n] \to X) \in (\bm{\Delta}  X)^{\op} } \Hom_{\cC}(F([n]), A) \\ &\cong \varprojlim\limits_{(\Delta[n] \to X) \in (\bm{\Delta}  X)^{\op} } \Hom_{\sSet}(\Delta[n], F_{\ast}(A))   \\ &\cong \Hom_{\sSet}(\varinjlim\limits_{(\Delta[n] \to X) \in \bm{\Delta}  X} \Delta[n], F_{\ast}(A))  \\ &\cong  \Hom_{\sSet}(X, F_{\ast}(A)),
		\end{align*}
		where the last equation follows from \cite[Proposition 15.1.20]{Hir03}.  So $F^{\ast}$ is left adjoint to $F_{\ast}$.
	\end{proof}
	
	\begin{expl}
		\begin{enumerate}
			\item Let $\bm{\Delta} \to \Cat$ be the functor defined by regarding $[n]$ as a posetal category: its objects are $0,1,\ldots n$ and $\Hom_{[n]}(k,\ell)$ has at most one element, and is non-empty if and only if $k\leq \ell$.  We write $P\colon \sSet \to \Cat$ and $B\colon \Cat \to \sSet$ for the associate left adjoint functor and right adjoint functor respectively.  The functor $B$ is called the nerve functor.  The simplicial set $B\cC=(X_n)$ is defined by sets $X_n\subset \mathrm{Ob}(\cC)^{[n]}$ of $(n+1)$-tuples $(C_0,\ldots,C_n)$ of objects of $\cC$ with morphisms $C_k\to C_\ell$ when $k\leq \ell$, which are compatible when $n$ varies; it is a fibrant simplicial set if and only if $\cC\in \Gpd$ (see \cite[Lemma \RNum{1}.3.5]{GJ09}).  In a word, for $B\cC$ to be fibrant, it must have the extension property with respect to inclusions of horns in $\Delta[n]$ ($\forall n\geq 1$).  For $n=2$, it amounts to saying that all homomorphisms in $\cC$ are invertible; for $n>2$, the extension condition is automatic (details in the reference above).  For $\cC\in \Cat$, we have $PB\cC\cong \cC$, so $\Hom_{\Cat}(\cC, \cD) \cong \Hom_{\sSet}(B\cC, B\cD)$ ($\forall \cC,\cD\in \Cat$).  Note that $B(\cC\times [1]) \cong B\cC \times \Delta[1]$ (product is the degreewise product); in consequence, when $\cC\in \Cat$ and $\cD \in \Gpd$, two functors $f,g\colon \cC\to\cD$ are naturally isomorphic if and only if $Bf$ and $Bg$ are homotopic.
			\item As a corollary of (1), we have $\Hom_{\Gpd}(GPX, H) \cong \Hom_{\sSet}(X, BH)$ for $X\in \sSet$ and $H \in \Gpd$, where $GPX$ is the free groupoid associated to $PX$.  We remark that $GPX$ and $\pi_{1}|X|$ (the fundamental groupoid of the geometric realization) are isomorphic in $\Ho(\Gpd)$ (see \cite[Theorem \RNum{3}.1.1]{GJ09}). 
		\end{enumerate}
	\end{expl}
	
	Recall that a functor between two model categories is called right Quillen if it preserves fibrations and trivial fibrations (fibrations which are weak equivalences).
	
	\begin{lem} 
		The nerve functor $B\colon \Gpd \to \sSet$ is fully faithful and takes fibrant values (Kan-valued).  Moreover, it is right Quillen.
	\end{lem}
	
	\begin{proof}
		For the first statement, we know by Example 2.2 that: $\Hom_{\Cat}(\cC, \cD) \cong \Hom_{\sSet}(B\cC, B\cD)$ ($\forall \cC,\cD\in \Cat$, hence the fully faithfulness. Moreover $B\cC$ is fibrant for $\cC$ a groupoid.
		
		For the second statement, note that $B$ obviously preserves weak equivalences; moreover, by definition, $Bf\colon BG\to BH$ is a fibration if and only if it has the right lifting property with respect to inclusions of horns in $\Delta[n]$, $\forall n\geq 1$ (see \cite[page 10]{GJ09}). For $n=1$ this means exactly that $f$ is a fibration, while for $n\geq2$ it's automatic (see the proof of \cite[Lemma \RNum{1}.3.5]{GJ09}).
	\end{proof}
	
	Let $A\in\Art_\cO$. Consider the group $G(A)$ of $A$-points of our reductive group scheme $G$.  Passing to homotopy categories, we get the isomorphism 
	\begin{align*}
	\Hom_{\Ho(\Gpd)}(\Gamma_{i}, G(A)) &\cong \Hom_{\Ho(\sSet)}(B\Gamma_{i}, BG(A))\\
	&\cong \pi_{0}\sHom_{\sSet}(B\Gamma_{i}, BG(A)).
	\end{align*}
	
	Let $X=(B\Gamma_{i})_{i}$ be the pro-simplicial set associated to the profinite group $\Gamma$.  We define $$\Hom_{\sSet}(X, -)=\varinjlim_{i} \Hom_{\sSet}(B\Gamma_{i}, -).$$  Then the Galois representation $\bar{\rho}\colon \Gamma \to G(k)$ gives rise to an element of $\Hom_{\sSet}(X, BG(k))$, which we also denote by $\bar{\rho}$.  In order to take into account the deformations of $\bar{\rho}$, we introduce the overcategory $\cM = \sSet/_{BG(k)}$ of pairs $(Y,\pi)$ where $Y$ is a simplicial set and $\pi\colon Y\to BG(k)$ is a morphism of simplicial sets. The category $\cM$ has a natural simplicial model category structure: the cofibrations, fibrations, weak equivalences and tensor products are those of $\sSet$ (see \cite[Lemma \RNum{2}.2.4]{GJ09} for the only nontrivial part of the statement).  When we consider $X\in \cM$, we specify the morphism $\bar\rho \colon X\to BG(k)$; similarly, when we consider $BG(A) \in \cM$ for $A\in \Art_{\O}$, we specify the natural projection $BG(A)\to BG(k)$.  For $X,Y\in \cM$, we can define an object of $\cM$ of enriched homomorphisms $\sHom_{\cM}(X,Y)$ for which $\sHom_{\cM}(X,Y)_n$ consists in the morphisms $X\times\bm{\Delta}[n]\to Y$ compatible to the projections to $BG(k)$.  Since $BG(A)\to BG(k)$ is a fibration, $BG(A)\in \cM$ is fibrant.  Similar to the discussion of the preceding paragraph, we have 
	\begin{equation}\label{D(A)}
	\cD(A) \cong \Hom_{\Ho(\cM)}(X, BG(A)) \cong \pi_{0}\sHom_{\cM}(X, BG(A))
	\end{equation}
	for $A\in \Art_{\O}$.  Note that $\sHom_{\cM}(X, BG(A))$ is the fiber over $\bar{\rho}$ of the fibration map $$\sHom_{\sSet}(X, BG(A)) \to \sHom_{\sSet}(X, BG(k)),$$ so it actually calculates the homotopy fiber of $\pi_{\bar{\rho}}$ (see \cite[Theorem 13.1.13 and Proposition 13.4.6]{Hir03}).
	
	When $\Gamma=G_{F,S}$, $S=S_p\cup S_\infty$ and $\bar{\rho}$ satisfies $(Ord_v)$ for $v\in S_{p}$, we reformulate the definition of the nearly ordinary deformation subfunctor $\cD^{\ord}\subset\cD$ as follows.  For each $v\in S_p$, we form $\Gamma_v=\varprojlim_i\Gamma_{i,v}$ where $\Gamma_v\to \Gamma$ induces morphisms $\Gamma_{i,v}\to\Gamma_i$ of finite groups.  Let $X_{v}=(B\Gamma_{i,v})_{i}$ be the pro-simplicial set associated.  For the fixed Borel subgroup $B$ of $G$, we have a natural cofibration $BB(A)\subset BG(A)$.  Recall that $\bar{g}_v\cdot 
	\bar\rho\vert_{\Gamma_v}\cdot \bar{g}_v^{-1}$ takes values in $B(k)$.  Let $\cD_v(A)$ be $\pi_0$ of the fiber over $	\bar\rho\vert_{\Gamma_v}$ of the fibration map 
	$$\sHom_{\sSet}(X_v, BG(A)) \to \sHom_{\sSet}(X_v, BG(k)),$$ 
	and let $\cD_v^{\ord}(A)$ be $\pi_0$ of the fiber over $\bar{g}_v\cdot 
	\bar\rho\vert_{\Gamma_v}\cdot \bar{g}_v^{-1}$ of the fibration map 
	$$\sHom_{\sSet}(X_v, BB(A)) \to \sHom_{\sSet}(X_v, BB(k)).$$ 
	Then there is a natural functorial inclusion $i_v$ of $\cD_v^{\ord}(A)$ into $\cD_v(A)$.  Let $\cD_{\loc}(A)=\prod_{v\in S_p}\cD_v(A)$ and $\cD_{\loc}^{\ord}(A)=\prod_{v\in S_p}\cD_v^{\ord}(A).$  There is a natural functorial map $\cD(A)\to \cD_{\loc}(A)$, resp. $\cD_{\loc}^{\ord}(A)\to \cD_{\loc}(A)$, induced by $\rho\mapsto (\rho\vert_{\Gamma_v})_{v\in S_p},$ resp. by $\prod_{v\in S_p} i_v$. 
	
	We define $\cD^{\ord}(A)$ as the fiber product $$\cD^{\ord}(A)=\cD(A)\times_{\cD_{\loc}(A)} \cD_{\loc}^{\ord}(A).$$
	
	\begin{lem}\label{ordinary reinterpretation}
		Suppose $(Reg_v)$ holds for each place $v\in S_p$.  Then the functor $\cD^{\ord}$ is isomorphic to the classical nearly ordinary deformation functor.
	\end{lem}
	
	\begin{proof}
		It follows easily from what precedes. See \cite{Cai20} or \cite{TU20}.
	\end{proof}

	\subsection{Simplicial reformulation of classical framed deformations}
	
	Let $\Gpd_{\ast}$ and $\sSet_{\ast}$ be the model categories of based groupoids and based simplicial sets (in other words, under categories ${}_{\ast} \backslash \Gpd$ and ${}_{\ast} \backslash \sSet$) respectively.  Then we have 
	\begin{displaymath}
	\Hom_{\Gp}(\Gamma_{i}, G(A)) \cong \Hom_{\Ho(\Gpd_{\ast})}(\Gamma_{i}, G(A)).
	\end{displaymath}
	Let $\cM_{\ast}$ be the over and under category ${}_{\ast} \backslash \sSet /_{BG(k)}$.  Note that $X$ and $BG(A)$ for $A\in \Alg_{\O}$ are naturally objects of $\cM_{\ast}$.  Proceeding as the unframed case, we see that 
	\begin{equation}\label{D(A)frame}
	\cD^{\square}(A) \cong \Hom_{\Ho(\cM_{\ast})}(X, BG(A)) \cong \pi_{0}\sHom_{\cM_{\ast}}(X, BG(A)).
	\end{equation}
	We remark that $\sHom_{\cM_{\ast}}(X, BG(A))$ is weakly equivalent to $\hofib_{\ast}(\sHom_{\cM}(X, BG(A)) \to \sHom_{\cM}(\ast, BG(A)))$, since $\sHom_{\cM}(X, BG(A)) \to \sHom_{\cM}(\ast, BG(A))$ is a fibration.

	\subsection{Derived deformation functors}
	
	We have defined the functor $\sHom_{\cM}(X, BG(-))$ from $\Art_{\O}$ to $\sSet$.  Our next goal is to extend this functor to simplicial Artinian $\O$-algebras over $k$, which we define below.
	
	Let $\sCR$ be the category of simplicial commutative rings (these are simplicial sets which are rings in all degrees and for which all face and degeneracy maps are ring homomorphisms).  A usual commutative ring $A$ can be regarded as an element of $\sCR$, which consists of $A$ on each simplicial degree with identity face and degeneracy maps.  In this way we regard $\O$ and $k$ as objects of $\sCR$.  With the natural reduction map $\O\to k$, the over and under category ${}_{\O} \backslash \sCR /_{k}$ has a simplicial model category stucture, such that the cofibrations, fibrations and weak equivalences are those of $\sCR$, and the tensor product of $A\in {}_{\O} \backslash \sCR /_{k}$ and $K\in \sSet$ is the pushout of $\O \leftarrow \O\otimes K \rightarrow A\otimes K$.  Note that degreewise surjective morphisms $A\to B$ are fibrations. 
	
	Since $\sCR$ is cofibrantly generated, any $A \in {}_{\O} \backslash\sCR$ admits a functorial cofibrant replacement $c(A)$: $$\O \hookrightarrow c(A) \stackrel{\sim}{\twoheadrightarrow} A.$$  Concretely, for any $n\geq 0$ the $\cO$-algebra	$c(A)_n$ is a suitable polynomial $\cO$-algebra mapping surjectively onto $A_n$.  The key property of the cofibrant replacement is that 
	
	-$c(A)$ is a cofibrant object and 
	
	-$c(A)\to A$ is a trivial fibration (a fibration which is a weak equivalence).
	
	Note that the functor $B\mapsto \sHom (c(A),B)$ commutes to weak equivalence (this is called homotopy invariance), while it is not necessarily the case of the functor $B\mapsto \sHom (A,B)$. 
	
	For $A\in {}_{\O} \backslash \sCR$, for any $i\geq 0$, $\pi_iA$ is a commutative group and $\bigoplus_i\pi_iA$ is naturally a graded $\cO$-algebra, hence a $\pi_0 A$-algebra (see \cite[Lemma 8.3.2]{Gil13}). 	
	
	\begin{de}
		The simplicial Artinian $\O$-algebras over $k$, which we denote by $\sArt$, is the full subcategory of $ {}_{\O} \backslash \sCR /_{k}$ consisting of objects $A \in {}_{\O} \backslash \sCR /_{k}$ such that:
		\begin{enumerate}
			\item $\pi_{0}A$ is Artinian local in the usual sense.
			\item $\pi_{\ast}A = \oplus_{i\geq 0}\pi_{i}A$ is finitely generated as a module over $\pi_{0}A$. 
		\end{enumerate}
	\end{de}
	
	Note that $\sArt$ is not a model category, and cofibrations, fibrations and weak equivalences in $\sArt$ are used to indicate those in ${}_{\O} \backslash \sCR /_{k}$.  Nevertheless, $\sArt$ is closed under weak equivalences since the definition only involves homotopy groups.  We also remark that every $A\in \sArt$ is fibrant since $A \to k$ is degreewise surjective.
	
	We define $\O_{N_{\bullet}G} \in \Alg_{\O}^{\bm{\Delta}}$ (i.e., a functor $\bm{\Delta}\to \Alg_{\O}$, also called a cosimplicial object in $\Alg_{\O}$) as follows: in codegree $p$ we have $\O_{N_{p}G} = \O_{G}^{\otimes p}$, and the coface and codegeneracy maps are induced from the comultiplication and the coidentity of the Hopf algebra $\O_{G}$ respectively.  Then for $A\in \Alg_{\O}$, the nerve $BG(A)$ is nothing but $\Hom_{\Alg_{\O}}(\O_{N_{\bullet}G}, A)$, with face and degeneracy maps induced by the coface and codegeneracy maps in $\O_{N_{\bullet}G}$.  When $A\in {}_{\O} \backslash \sCR$, the na\"\i ve analogy is the diagonal of the bisimplicial set $([p],[q]) \mapsto \Hom_{\Alg_{\O}}(\O_{N_{p}G}, A_{q})$ (recall that the diagonal of a bisimplicial set is a simplicial set model for its geometric realization).  However, we need to make some modifications using cofibrant replacements to ensure the homotopy invariance.
	
	\begin{de}
		\begin{enumerate}
			\item For $A\in {}_{\O} \backslash \sCR$, we define $\Bi(A)$ to be the bisimplicial set $$([p],[q])\mapsto  \Hom_{{}_{\O} \backslash \sCR}(c(\O_{N_{p}G}),A^{\Delta[q]}),$$ with face and degeneracy maps induced by the coface and codegeneracy maps in $\O_{N_{\bullet}G}$ and the face and degeneracy maps in $A^{\Delta[\bullet]}$.
			\item The diagonal of $\Bi(A)$, which is denoted by $\diag \Bi(A)$, is the simplicial set induced from the diagonal embedding $\bm{\Delta}^{\op} \to \bm{\Delta}^{\op} \times \bm{\Delta}^{\op} \xrightarrow{\Bi(A)} \Set$.  
		\end{enumerate}
	\end{de}
	
	When $A$ is an $\O$-algebra regarded as a constant object in ${}_{\O} \backslash \sCR$, we have $$\Bi(A)_{p,q} = \Hom_{{}_{\O} \backslash \sCR}(c(\O_{N_{p}G}),A^{\Delta[q]}) \cong \Hom_{\Alg_{\O}}(\O_{N_{p}G},A),$$ where the latter isomorphism is because the constant embedding functor is right adjoint to $\pi_{0}\colon {}_{\O} \backslash \sCR \to \Alg_{\O}$.  Hence $\Bi(A)$ is just a disjoint union of copies of $BG(A)$ in index $q$.  In particular, for $A\in \sArt$ there is a natural map $\Bi(A)_{\bullet,q} \to BG(k)$ for each $q\geq 0$, so we may regard $\Bi(A) \in \cM^{\bm{\Delta}^{\op}}$ via the association $[q]\mapsto \Bi(A)_{\bullet, q}$ (recall that $\cM$ is the overcategory $\sSet/_{BG(k)}$), and $\diag \Bi(A)$ is an object of $\cM$).  Recall that any morphism $X\to Y$ in $\sSet$ admits a functorial factorisation $$X \stackrel{\sim}{\hookrightarrow} \widetilde{X} \twoheadrightarrow Y$$ into a trivial cofibration and a fibration.
	
	\begin{de}\label{BG}
		For $A\in \sArt$, the simplicial set $\BG(A)$ is defined by the functorial trivial cofibration-fibration factorization $\diag \Bi(A) \stackrel{\sim}{\hookrightarrow} \BG(A) \twoheadrightarrow BG(k)$.
	\end{de}
	
	It's clear that $\BG\colon \sArt \to \cM$ defines a functor.  If $A\in \Art_{\O}$ is regarded as a constant simplicial ring, then $\diag \Bi(A) = BG(A) \twoheadrightarrow BG(k)$ is a fibration, so $BG(A)$ is a strong deformation retract of $\BG(A)$ in $\cM$ (see \cite[Definition 7.6.10]{Hir03}).  In particular, these two are indistinguishable in our applications.
	
	\begin{rem}
		Our $\BG(A)$ is weakly equivalent to the simplicial set $\Ex^{\infty}\diag \Bi(A)$ which is the definition chosen in \cite[Definition 5.1]{GV18}.  There is a slight difference: we want to emphasize the fibration $\BG(A) \twoheadrightarrow BG(k)$, so that it's more convenient to handle the homotopy pullbacks.
	\end{rem}
	
	As mentioned above, the reason for taking cofibrant replacements is:
	
	\begin{lem}
		If $A \to B$ is a weak equivalence, then so is $\BG(A) \to \BG(B)$.
	\end{lem}
	
	\begin{proof}
		If $A \to B$ is a weak equivalence, then $\sHom_{{}_{\O} \backslash \sCR}(c(\O_{N_{p}G}),A) \to \sHom_{{}_{\O} \backslash \sCR}(c(\O_{N_{p}G}),B)$ is a weak equivalence for each $p\geq 0$, so is $\diag \Bi(A) \to \diag \Bi(B)$ (see \cite[Theorem 15.11.11]{Hir03}), and so is $\BG(A) \to \BG(B)$.
	\end{proof}
	
	\begin{de}
		\begin{enumerate}
			\item The derived universal deformation functor $s\cD\colon \sArt \to \sSet$ is defined by 
			$$s\cD(A)=\sHom_{\cM}(X, \BG(A)). $$
			\item The derived universal framed deformation functor $s\cD^{\square} \colon \sArt \to \sSet$ is defined by 
			$$s\cD^{\square}(A)=\hofib_{\ast} (s\cD(A) \to \sHom_{\cM}(\ast, \BG(A))). $$
		\end{enumerate}
	\end{de}
	
	\begin{rem}
		In \cite[Definition 5.4]{GV18}, the derived universal deformation functor is defined by $$s\cD(A) = \hofib_{\bar{\rho}}(\sHom_{\sSet}(X, \Ex^{\infty}\diag \Bi(A)) \to \sHom_{\sSet}(X, BG(k))).$$  Since $\Ex^{\infty}\diag \Bi(A))$ and $\BG(A)$ are weakly equivalent fibrant simplicial sets, $\sHom_{\sSet}(X, \Ex^{\infty}\diag \Bi(A))$ is weakly equivalent to $\sHom_{\sSet}(X, \BG(A))$.  But $\sHom_{\sSet}(X, \BG(A)) \to \sHom_{\sSet}(X, BG(k))$ is a fibration, so $\sHom_{\cM}(X, \BG(A))$ is weakly equivalent to the homotopy fiber.
	\end{rem}
	
	When $\Gamma=G_{F,S}$, $S=S_p\cup S_\infty$ and $\bar{\rho}$ satisfies $(Ord_v)$ for $v\in S_{p}$, we can define for each $v\in S_p$ a functor $s\cD_v\colon\sArt\to \sSet$ as $A\mapsto \sHom_{\sSet/_{BG(k)}}(X_{v}, \BG(A))$, and a functor $s\cD_v^{\ord}\colon\sArt\to \sSet$ as $A\mapsto \sHom_{\sSet/_{BB(k)}}(X_{v}, \BB(A))$.  Let $s\cD_{\loc}=\prod_{v\in S_p}s\cD_v$ and let $s\cD_{\loc}^{\ord}=\prod_{v\in S_p}s\cD_v^{\ord}$.  Define $s\cD^{\ord}$ as the homotopy fiber product $$s\cD^{\ord}=s\cD\times^h_{s\cD_{\loc}}s\cD_{\loc}^{\ord}. $$
	
	\begin{de}
		Let $\cF\colon \sArt \to \sSet$ be a functor.  We say $\cF$ is formally cohesive if it satisfies the following conditions:
		\begin{enumerate}
			\item $\cF$ is homotopy invariant ({\it i.e.} preserves weak equivalences).
			\item Suppose that 
			\begin{displaymath}
			\xymatrix{
				A \ar[r]\ar[d] & B \ar[d] \\
				C \ar[r] & D}
			\end{displaymath}
			is a homotopy pullback square with at least one of $B \to D$ and $C \to D$ degreewise surjective, then
			\begin{displaymath}
			\xymatrix{
				\cF(A) \ar[r]\ar[d] & \cF(B) \ar[d] \\
				\cF(C) \ar[r] & \cF(D)}
			\end{displaymath}
			is a homotopy pullback square.
			\item $\cF(k)$ is contractible.
		\end{enumerate}
	\end{de}
	
	We summarize our preceding discussions:
	
	\begin{pro}
		The functors $s\cD$, $s\cD^{\square}$, $s\cD_{v}^{?}$ (here $?=\emptyset$ or $\ord$) and $s\cD^{\ord}$ are all formally cohesive.  
	\end{pro}
	
	\begin{proof}
		We first verify three conditions in the above definition for $s\cD$:
		\begin{enumerate}
			\item If $A \to B$ is a weak equivalence, then $\BG(A)\to \BG(B)$ is a weak equivalence between fibrant objects in $\cM$, so $\sHom_{\cM}(X, \BG(A)) \to \sHom_{\cM}(X, \BG(B))$ is also a weak equivalence.
			\item First we show that 
			\begin{displaymath}
			\xymatrix{
				\BG(A) \ar[r]\ar[d] & \BG(B) \ar[d] \\
				\BG(C) \ar[r] & \BG(D)}
			\end{displaymath}
			is a homotopy pullback square in $\cM$.  Note that regarding the above diagram as a diagram in $\sSet$ doesn't affect the homotopy pullback nature.  By \cite[Lemma 4.31]{GV18}, it suffices to check:
			\begin{enumerate}
				\item the functor $\Omega\BG\colon \sArt \to \sSet$ preserves homotopy pullbacks, and 
				\item $\pi_{1}\BG(C) \to \pi_{1}\BG(D)$ is surjective whenever $C \to D$ is degreewise surjective.
			\end{enumerate}
			Part (a) follows from \cite[Lemma 5.2]{GV18}, and part (b) follows from \cite[Corollary 5.3]{GV18}.
			
			Since small filtered colimits of simplicial sets preserve homotopy pullbacks, we may suppose the pro-object $X$ lies in $\cM$.  Then $\sHom_{\cM}(X, -) \colon \cM \to \sSet$ is a right Quillen functor, hence its right derived functor commutes with homotopy pullbacks in the homotopy categories.  But we are dealing with fibrant objects, so in the homotopy category $\sHom_{\cM}(X, -)$ is isomorphic to its right derived functor.  The conclusion follows.
			\item It's clear that $s\cD(k)$ is contractible.
		\end{enumerate}
		
		The same argument applies for $A \to \sHom_{\cM}(\ast, \BG(A))$.  So $s\cD^{\square}$ is formally cohesive because it is the homotopy pullback of formally cohesive functors.
		
		In the nearly ordinary case, we may replace $X$ by $X_{v}$ and replace $G$ by $B$ and the same argument applies.  Hence $s\cD_{v}^{?}$  ($?=\emptyset$ or $\ord$) is formally cohesive.  Since $s\cD^{\ord}$ is the homotopy limits of formally cohesive functors, it is also formally cohesive.
	\end{proof}

	\subsubsection{Modifying the center}\label{modifying the center}
	
	None of these functors cannot be pro-representable unless $G$ is of adjoint type.  If $G$ has a non trivial center $Z$, we need a variant $s\cD_Z$, resp. $s\cD_Z^{\ord}$, of the functor $s\cD$, resp. of $s\cD^{\ord}$,  in order to allow pro-representability. For this modification, we follow \cite[Section 5.4]{GV18}.  For a classical ring $A\in \Art$, we have a short exact sequence $$1\to Z(A)\to G(A)\to PG(A)\to 1.$$  It yields a fibration sequence $BG(A)\to BPG(A)\to B^2Z(A)$.  Indeed, given a simplicial group $H$ and a simplicial sets $X$ with a left $H$-action, we can form the bar construction $N_\bullet(\ast,H,X)$ at each simplicial degree (see \cite[Example 3.2.4]{Gil13}), which gives the bisimplicial set $([p],[q]) \mapsto  H_p^q\times X_p=: N_q (\ast, H_p, X_p)$.  Consider the action $Z(A)\times G(A)\to G(A)$, and the corresponding simplicial action $N_pZ(A)\times N_pG(A)\to N_pG(A)$ (note that $N_\bullet Z(A)$ is a simplicial group because $Z(A)$ is abelian).  We identify for each $p\geq 0$, $$BG(A)_p=N_p(\ast,\ast,N_pG(A)),$$ $$BPG(A)_p=N_p(\ast,N_pZ(A),N_pG(A)),$$ and we put
	$$B^2Z(A)_p=N_p(\ast,N_pZ(A),\ast)$$ (with diagonal face and degeneracy maps).  The desired fibration is given by the canonical morphisms of simplicial sets which in degree $p$ are: 
	$$N_p(\ast,\ast,N_pG(A))\to N_p(\ast,N_pZ(A),N_pG(A))\to N_p(\ast,N_pZ(A),\ast).$$  
	
	Let us generalize this to $A\in \sArt$.  For this, we note first that $BPG(A)$ can also be defined as the functorial fibrant replacement of $\diag (N)$ where $N$ is the trisimplicial set associated to $(p,q,r)\mapsto N_q(\ast,N_pZ(A_r),N_p(G(A_r))$ (replacing $\cO_{NpG(A_r)}$ by its functorial cofibrant replacement as above). 

	Then, we define  $B^2Z(A)$ as the functorial fibrant replacement of $\diag (N^\prime)$ where $N^\prime$ is the trisimplicial set associated to $(p,q,r)\mapsto N_q(\ast,N_pZ(A_r),\ast)$ (replacing $\cO_{NpG(A_r)}$ by its functorial cofibrant replacement as above).  The obvious system of maps $N_q(\ast,N_pZ(A_r),N_pG(A_r))\to N_q(\ast,N_pZ(A_r),\ast)$ gives the desired map $$\mathcal{B}PG(A)\to B^2Z(A).$$

	The functor $s\cD_Z \colon \sArt \to \sSet$ is defined by the homotopy pullback square (here for simplicity we use $\cM$, but the base maps are those induced from $\BG(k) \to \mathcal{B}PG(k) \to B^2Z(k)$)
	\begin{displaymath}
	\xymatrix{
		s\cD_Z(A) \ar[r]\ar[d] & \sHom_{\cM}(\ast,B^2Z(A)) \ar[d] \\
		\sHom_{\cM}(X,\mathcal{B}PG(A)) \ar[r] & \Hom_{\cM}(X,B^2Z(A))}
	\end{displaymath}
	Then $s\cD_Z$ is formally cohesive becasue it is the homotopy pullback of formally cohesive functors.  Observe that $s\cD_Z$ and $s\cD$ coincide when $Z$ is trivial.  
	
	\begin{rem} 
		\begin{enumerate}
			\item We'll see later that $s\cD_Z$ is pro-representable, under the assumption $H^{0}(\Gamma, \g_k) = \z_k$.
			\item In the nearly ordinary case, one defines similarly $s\cD_{\loc,Z}= \prod_{v\in S_{p} }s\cD_{v,Z}$ and $s\cD_{\loc,Z}^{\ord} =\prod_{v\in S_{p} } s\cD^{\ord}_{v,Z}$.  Note that the construction for $s\cD_Z$ is functorial in $X$ and $G$, we can form the homotopy pullback $$s\cD_{Z}^{\ord}=s\cD_{Z}\times^h_{s\cD_{\loc,Z}}s\cD_{\loc,Z}^{\ord}. $$
			All these functors are formally cohesive.  We'll see later that $s\cD_{Z}^{\ord}$ is pro-representable, under the assumption $H^{0}(\Gamma, \g_k) = \z_k$.
		\end{enumerate}
	\end{rem}
	
	\begin{pro}\label{simplord}
		When $A$ is homotopy discrete, we have $\pi_{0}s\cD_{Z}(A) \cong \cD(\pi_{0}A)$ and $\pi_{0}s\cD_{v,Z}^{?}(A) \cong \cD^{?}_{v}(\pi_{0}A)$ (here $?=\emptyset$ or $\ord$).  If in addition $(Reg_v)$ holds for each $v\in S_{p}$, then $\pi_{0}s\cD_{Z}^{\ord}(A) \cong \cD^{\ord}(\pi_{0}A)$.
	\end{pro}
	
	\begin{proof}
		We may suppose $A\in \Art_{\O}$ by the formal cohesiveness.
		
		From the definition of $s\cD_{Z}$ it follows that we have a natural fibration sequence $$s\cD(A)\to s\cD_Z(A) \to \sHom_{\cM}(\ast, B^2Z(A)).$$  Since $\pi_{i}\sHom_{\cM}(\ast, B^2Z(A))$ vanishes for $i\neq 2$, we have $\pi_0 s\cD_{Z}(A)=\pi_0 s\cD(A)$.  By Equation \ref{D(A)} of section \ref{subsectReform}, we have $\pi_0 s\cD(A)=\cD(A)$, hence also $\pi_0 s\cD_{Z}(A)=\cD(A)$.
		
		By applying the same arguement with $X$ replaced by $X_v$ and $G$ replaced by $B$ when necessary, we obtain $\pi_{0}s\cD_{v,Z}^{?}(A) \cong \cD^{?}_{v}(A)$ ($?=\emptyset$ or $\ord$).  
		
		We have the exact sequence 
		\begin{align*}
		\pi_{1} s\cD_{Z}(A) \oplus (\bigoplus_{v \in S_{p}} \pi_{1}s\cD_{v,Z}^{\ord}(A) ) \to \bigoplus_{v \in S_{p}} \pi_{1}s\cD_{v,Z}(A) \\ \to \pi_{0}s\cD_{Z}^{\ord}(A) \to \pi_{0} s\cD_{Z}(A) \oplus (\bigoplus_{v \in S_{p}} \pi_{0}s\cD_{v,Z}^{\ord}(A) ) \to \bigoplus_{v \in S_{p}} \pi_{0}s\cD_{v,Z}(A)
		\end{align*}
		We will see later (\lemref{calculation}) that $s\cD_{v}(A)$ is weakly equivalent to $\holim_{\bm{\Delta}X} \hofib_{\ast}(BG(A) \to BG(k))$, and (by \lemref{calculation 2}) $\pi_1s\cD_{v}(A) \cong H^0(\Gamma_v, \widehat{G}(A))$.  Similarly $\pi_1 s\cD^{\ord}_{v}(A) \cong H^0(\Gamma_v, \widehat{B}(A))$.  
		
		By the assumption $(Reg_v)$ and Artinian induction, the map $\pi_{1}s\cD_{v}^{\ord}(A) \to \pi_{1}s\cD_{v}(A)$ is an isomorphism, and so is $\pi_{1}s\cD_{v,Z}^{\ord}(A) \to \pi_{1}s\cD_{v,Z}(A)$.  We deduce that $\pi_{0}s\cD_{Z}^{\ord}(A)$ is the kernel of $\cD(A) \oplus (\bigoplus_{v \in S_{p}} \cD_{v}^{\ord}(A) ) \to \bigoplus_{v \in S_{p}} \cD_{v}(A)$, which is isomorphic to $\cD^{\ord}(A)$ by \lemref{ordinary reinterpretation}.
	\end{proof}

	\section{Pseudo-deformation functors}

	\subsection{Classical pseudo-characters and functors on $\FFS$}
	
	Recall the notion of a (classical) $G$-pseudo-character due to V. Lafforgue (see \cite[D\'efinition-Proposition 11.3]{Laf18} and \cite[Definition 4.1]{BHKT19}):
	
	\begin{de}
		Let $A$ be an $\O$-algebra.  A $G$-pseudo-character $\Theta$ on $\Gamma$ over $A$ is a collection of $\O$-algebra morphisms $\Theta_{n}\colon \O_{N_{n}G}^{\ad G} \to \Map(\Gamma^{n}, A)$ for each $n\geq 1$, satisfying the following conditions:
		\begin{enumerate}
			\item For each $n,m \geq 1$ and for each map $\zeta\colon \{1,\dots, n \} \to \{1,\dots, m \} $, $f\in \O_{N_{m}G}^{\ad G}$, and $\gamma_{1},\dots, \gamma_{m} \in \Gamma$, we have 
			\begin{displaymath}
			\Theta_{m}(f^{\zeta})(\gamma_{1},\dots, \gamma_{m}) = \Theta_{n}(f)(\gamma_{\zeta(1)},\dots, \gamma_{\zeta(n)}),
			\end{displaymath}
			where $f^{\zeta}(g_{1},\dots, g_{m}) = f(g_{\zeta(1)},\dots, g_{\zeta(n)})$.
			\item For each $n\geq 1$, for each $\gamma_{1},\dots, \gamma_{n+1} \in \Gamma$, and for each $f\in \O_{N_{n}G}^{\ad G}$, we have
			\begin{displaymath}
			\Theta_{n+1}(\hat{f})(\gamma_{1},\dots, \gamma_{n+1}) = \Theta_{n}(f)(\gamma_{1},\dots, \gamma_{n-1},\gamma_{n}\gamma_{n+1}),
			\end{displaymath}
			where  $\hat{f}(g_{1},\dots, g_{n+1}) = f(g_{1},\dots, g_{n-1}, g_{n}g_{n+1})$.
		\end{enumerate}
		We denote by $\PsCh(A)$ the set of pseudo-characters over $A$.
	\end{de}
	We want to give a simplicial reformulation of this notion.	As a first step, following \cite{Weid18}, let us consider $\FS$ the category of finite sets and $\FFS$ be the category of finite free semigroups.  For any finite set $X$, let $M_X$ be the finite free semigroup generated by $X$; we have $\Gamma^X=\Hom_{\semGp}(M_X,\Gamma)$ and $ G^X=\Hom_{\semGp}(M_X,G)$.  For a semigroup $M\in \FFS$, note that $\Hom_{\semGp}(M_X,G)$ is a group scheme, so, we can define a covariant functor $\FFS\to \Alg_\cO$, $M\mapsto \cO_{\Hom_{\semGp}(M,G)}$.  We can also define the covariant functor $M\mapsto \Map(\Hom_{\semGp}(M,\Gamma),A)$.  These functors on $\FFS$ extend canonically those defined on the category $\FS$ by $X\mapsto \cO_{G^X}$ and $X\mapsto \Map(\Gamma^X,A)$.  Moreover, the natural transformation $$\O_{G^X}^{\ad G} \to \Map(\Gamma^X,A)$$ extends uniquely to a natural transformation of functors on $\FFS$.  Actually, there are several useful functors on $\FFS$; by the canonical extension from $\FS$ to $\FFS$ mentioned above, it is enough to define them on the objects $[n]$, as in \cite[Example 2.4 and Example 2.5]{Weid18}: 
	\begin{enumerate}
		\item The association $[n] \mapsto \Gamma^{n}$ defines an object $\Gamma^{\bullet}\in \Set^{\FFS^{\op}}$.
		\item For $A\in \Alg_{\O}$, the association $[n] \mapsto \Map(\Gamma^{n}, A)$ defines an object $\Map(\Gamma^{\bullet}, A) \in \Alg_{\O}^{\FFS}$.
		\item The association $[n] \mapsto \O_{N_{n}G}^{\ad G}$ defines an object $\O_{N_{\bullet}G}^{\ad G} \in \Alg_{\O}^{\FFS}$.
		\item Let $G^{n}\dslash G = \Spec(\O_{N_{n}G}^{\ad G})$.  Then for $A\in \Alg_{\O}$, the association $[n] \mapsto (G^{n}\dslash G)(A)$ defines an object $(G^{\bullet}\dslash G)(A) \in \Set^{\FFS^{\op}}$.
	\end{enumerate}
		
	As noted in \cite[Theorem 2.12]{Weid18}, one sees that a $G$-pseudo-character $\Theta$ of $\Gamma$ over $A$ is exactly a natural transformation from $\O_{N_{\bullet}G}^{\ad G}$ to $\Map(\Gamma^{\bullet}, A)$ (we call these natural transformations $\Alg_{\O}^{\FFS}$-morphisms).
	
	\begin{lem}
		For $A\in \Alg_{\O}$, there is a bijection between $\PsCh(A)$ and $\Hom_{\Set^{\FFS^\op}}(\Gamma^{\bullet},(G^{\bullet}\dslash G)(A))$.
	\end{lem}
	
	\begin{proof}
		It suffices to note that there is a bijection between $\Set^{\FFS^\op}$-morphisms $\Gamma^{\bullet} \to (G^{\bullet}\dslash G)(A)$ and $\Alg_{\O}^{\FFS}$-morphisms $\O_{N_{\bullet}G}^{\ad G}\to \Map(\Gamma^{\bullet}, A)$.
	\end{proof}
	
	For an algebraically closed field $A$ and a (continuous) homomorphism $\rho\colon \Gamma \to G(A)$, we say that $\rho$ is $G$-completely reducible if any parabolic subgroup containing $\rho(\Gamma)$ has a Levi subgroup containing $\rho(\Gamma)$.  Recall the following results in \cite[Section 4]{BHKT19}:
	
	\begin{thm}
		\begin{enumerate}
			\item \cite[Theorem 4.5]{BHKT19}  Suppose that $A \in \Alg_{\O}$ is an algebraically closed field.  Then we have a bijection between the following two sets:
			\begin{enumerate}
				\item The set of $G(A)$-conjugacy classes of $G$-completely reducible group homomorphisms $\rho\colon \Gamma \to G(A)$,
				\item The set of pseudo-characters over $A$.
			\end{enumerate}
			\item \cite[Theorem 4.10]{BHKT19} Fix an absolutely $G$-completely reducible representation $\bar\rho\colon \Gamma \to G(k)$, and suppose further that the centralizer of $\bar\rho$ in $G_{k}^{\ad}$ is scheme-theoretically trivial.  Let $\bar\Theta$ be the pseudo-character, which regarded as an element of $\Hom_{\Set^{\FFS^\op}}(\Gamma^{\bullet},(G^{\bullet}\dslash G)(k))$, is induced from $(\gamma_{1},\dots,\gamma_{n})\mapsto (\bar\rho(\gamma_{1}), \dots,\bar\rho(\gamma_{n}))$.  Let $A\in \Art_{\O}$.  Then we have a bijection between the following two sets:
			\begin{enumerate}
				\item The set of $\widehat{G}(A)$-conjugacy classes of group homomorphisms $\rho\colon \Gamma \to G(A)$ which lift $\bar\rho$,
				\item The set of pseudo-characters over $A$ which reduce to $\bar\Theta$ modulo $\m_{A}$.
			\end{enumerate}
		\end{enumerate}
	\end{thm}
	
	Note that there are similarities between $\Set^{\FFS^\op}$ and $\Set^{\bm{\Delta}^\op} = \sSet$.  In the following, we shall prove similar results with $\Set^{\FFS^\op}$ replaced by $\sSet$.

	\subsection{Classical pseudo-characters and simplicial objects}
	
	Recall that on $\O_{N_{\bullet}G}$ there are natural coface and codegeneracy maps, and we can regard $\O_{N_{\bullet}G}$ as an object in $\Alg_{\O}^{\bm{\Delta}}$ ({\it i.e.} a cosimplicial $\cO$-algebra).  The adjoint action of $G$ on $G^{\bullet}$ induces an action of $G$ on $\O_{N_{\bullet}G}$, which obviously commutes with the coface and codegeneracy maps.  In consequence, $\O_{N_{\bullet}G}^{\ad G}$ is well-defined in $\Alg_{\O}^{\bm{\Delta}}$.  
	
	\begin{de}
		We define the functor $\bar{B}G \colon \Alg_{\O} \to \sSet$ by associating $A\in \Alg_{\O}$ to $\Hom_{\Alg_{\O}}(\O_{N_{\bullet}G}^{\ad G}, A)$ with face and degeneracy maps induced from the coface and codegeneracy maps in $\O_{N_{\bullet}G}^{\ad G}$.
	\end{de}
	
	Note that the inclusion $\O_{N_{\bullet}G}^{\ad G} \to \O_{N_{\bullet}G}$ gives a natural transformation $BG \to \bar{B}G$.

	\subsubsection{Algebraically closed field}
	
	Let $A \in \Alg_{\O}$ be an algebraically closed field.  We would like to characterize the elements of $\Hom_{\sSet}(B\Gamma, \bar{B}G(A))$.  They correspond to the quasi-homomorphisms, which we define below.
	
	\begin{de}
		Let $\Gamma$ and $G$ be two groups.  We say a map $\rho\colon \Gamma \to G$ is a quasi-homomorphism if there exists a a map $\phi \colon \Gamma \to G$ such that $\rho(x)^{-1}\rho(xy) = \phi(x) \rho(y) \phi(x)^{-1}$ for any $x,y\in \Gamma$. 
	\end{de}
	
	Obviously a group homomorphism is a quasi-homomorphism.  Note that every quasi-homomorphism preserves the identity, and the set of quasi-homomorphisms is closed under $G$-conjugations.
	
	\begin{rem}
		A quasi-homomorphism can fail to be a group homomorphism.  We can construct a quasi-homomorphism as follows: let $\sigma \colon \Gamma \to G$ be a group homomorphism, let $\phi \colon \Gamma \to Z(\sigma(\Gamma))$ be a group homomorphism and let $g\in G$, then $\rho(x) = g^{-1} \sigma(x) \phi(x)g\phi(x)^{-1}$ is a quasi-homomorphism.  Such $\rho$ is not necessarily a group homomorphism, an example could be the following: take $G=H\times H$, $\sigma \colon \Gamma \to H \times \{\e\} $ and $\phi \colon \Gamma \to \{\e\} \times H$, and choose $g$ such that $g\notin Z(\phi(\Gamma))$.
	\end{rem}
	
	\begin{lem}\label{quasi-hom}
		Let $\rho$ be a quasi-homomorphism and let $\phi$ as above.  Then the map $\phi$ induces a group homomorphism $\Gamma \to G/Z(\rho(\Gamma))$ which doesn't depend on the choice of $\phi$.
	\end{lem}
	
	\begin{proof}
		For $x,y,z\in \Gamma$, we have
		\begin{align*}
		\phi(xy) \rho(z) \phi(xy)^{-1} &= \rho(xy)^{-1} \rho(xyz) \\ &= (\phi(x) \rho(y) \phi(x)^{-1})^{-1} (\phi(x) \rho(yz) \phi(x)^{-1}) \\&= \phi(x) \rho(y)^{-1}\rho(yz) \phi(x)^{-1}\\ &= \phi(x) \phi(y) \rho(z)  \phi(y)^{-1}\phi(x)^{-1}.
		\end{align*}
		Hence $\phi(xy)^{-1}\phi(x) \phi(y) \in Z(\rho(\Gamma))$ for any $x,y\in \Gamma$, and $\phi$ induces a group homomorphism $\Gamma \to G/Z(\rho(\Gamma))$.  For any other choice $\phi_{1}$ such that $\rho(x)^{-1}\rho(xy) = \phi_{1}(x) \rho(y) \phi_{1}(x)^{-1}$, we see $\phi_{1}^{-1}(x)\phi(x) \in Z(\rho(\Gamma))$, and the conclusion follows.
	\end{proof}
	
	\begin{lem}
		Suppose that $A \in \Alg_{\O}$ is an algebraically closed field.  Let $f\in \Hom_{\sSet}(B\Gamma, \bar{B}G(A))$.  Then we can associate a quasi-homomorphism $\rho\colon \Gamma \to G(A)$ to $f$ such that $f$ sends $(\gamma_{1},\dots, \gamma_{n}) \in B\Gamma_{n}$ to the class in $\bar{B}G(A)_{n}$ represented by $(\rho(\prod_{j=1}^{i-1}\gamma_{j})^{-1}\rho(\prod_{j=1}^{i}\gamma_{j}))_{i=1,\dots, n}$.
	\end{lem}
	
	\begin{proof}
		For each $n\geq 1$ and $\underline{\gamma} =(\gamma_{1},\dots, \gamma_{n}) \in \Gamma^{n}$, we choose a representative $T(\underline{\gamma}) = (g_{1}, \dots, g_{n}) \in G(A)^{n}$ of $f(\underline{\gamma})$ with closed orbit, note that any other representative with closed orbit is conjugated to $(g_{1}, \dots, g_{n})$.  Let $H(\underline{\gamma})$ be the Zariski closure of the subgroup of $G(A)$ generated by the entries of $T(\underline{\gamma})$.  Let $n(\underline{\gamma})$ be the dimension of a parabolic $P\subseteq G_{A}$ minimal among those containing $H(\underline{\gamma})$, we see $n(\underline{\gamma})$ is independent of the choice of $P$.  Let $N = \sup_{n\geq 1, \underline{\gamma} \in \Gamma^{n}}n(\underline{\gamma})$.  We fix a choice of $\underline{\delta} = (\delta_{1}, \dots, \delta_{n})$ satisfying the following conditions:
		\begin{enumerate}
			\item $n(\underline{\delta})=N$.
			\item For any $\underline{\delta}^{\prime} \in \Gamma^{n^{\prime}}$ satisfying (1), we have $\dim Z_{G_{A}}(H(\underline{\delta})) \leq \dim Z_{G_{A}}(H(\underline{\delta}^{\prime}))$.
			\item For any $\underline{\delta}^{\prime} \in \Gamma^{n^{\prime}}$ satisfying (1) and (2), we have $\# \pi_{0} (Z_{G_{A}}(H(\underline{\delta}))) \leq \# \pi_{0}( Z_{G_{A}}(H(\underline{\delta}^{\prime})))$. 
		\end{enumerate}
		Write $T(\underline{\delta}) = (h_{1}, \dots, h_{n})$.  As in the proof of \cite[Theorem 4.5]{BHKT19}, we have the following facts:
		\begin{enumerate}
			\item For any $(\gamma_{1}, \dots, \gamma_{m}) \in \Gamma^{m}$, there exists a unique tuple $(g_{1}, \dots, g_{m})\in G(A)^{m}$ such that $(h_{1}, \dots, h_{n}, g_{1}, \dots, g_{m})$ is conjugated to $T(\delta_{1}, \dots, \delta_{n},\gamma_{1}, \dots, \gamma_{m})$.  
			\item Let $(h_{1}, \dots, h_{n}, g_{1}, \dots, g_{m})$ be as above.  Any finite subset of the group generated by $(h_{1}, \dots, h_{n}, g_{1}, \dots, g_{m})$ which contains $(h_{1}, \dots, h_{n})$ has a closed orbit.
		\end{enumerate}
		We define $\rho(\gamma)$ to be the unique element such that $(h_{1}, \dots, h_{n}, \rho(\gamma))$ is conjugated to $T(\delta_{1}, \dots, \delta_{n},\gamma)$.  
		
		Suppose for $\gamma_{1},\dots,\gamma_{m} \in \Gamma$, the unique tuple conjugated to $T(\delta_{1}, \dots, \delta_{n},\gamma_{1}, \dots, \gamma_{m})$ is $(h_{1}, \dots, h_{n}, g_{1}, \dots, g_{m})$.  Consider the following diagram, where the horizontal arrows are compositions of face maps:
		\begin{displaymath}
		\xymatrix{
			(\delta_{1}, \dots, \delta_{n},\gamma_{1}, \dots, \gamma_{m})\ar[r]\ar[d]  & (h_{1},\dots, h_{n}, g_{1},\dots, g_{m}) \ar[d] \\
			(\delta_{1}, \dots, \delta_{n},\prod_{j=1}^{i}\gamma_{j}) \ar[r] & (h_{1},\dots, h_{n}, \prod_{j=1}^{i}g_{j})
		}
		\end{displaymath}
		Since $(h_{1},\dots, h_{n}, \prod_{j=1}^{i}g_{j})$ has a closed orbit and is a pre-image of $f(\delta_{1}, \dots, \delta_{n},\prod_{j=1}^{i}\gamma_{j})$, we have $\prod_{j=1}^{i} g_{j} = \rho(\prod_{j=1}^{i}\gamma_{j})$, and $g_{i} =\rho(\prod_{j=1}^{i-1}\gamma_{j})^{-1}\rho(\prod_{j=1}^{i}\gamma_{j})$ $(\forall i=1,\dots,m)$.
		
		Let $x,y\in\Gamma$.  Then the element in $G(A)^{2n+2}$ associated to $(\delta_{1},\dots, \delta_{n},x, \delta_{1},\dots, \delta_{n},y)$ is $$(h_{1},\dots, h_{n}, \rho(x), \rho(x)^{-1}\rho(x\delta_{1}),\dots, \rho(x\prod_{j=1}^{n-1}\delta_{j})^{-1}\rho(x\prod_{j=1}^{n}\delta_{j}), \rho(x\prod_{j=1}^{n}\delta_{j})^{-1}\rho(x\prod_{j=1}^{n}\delta_{j} \cdot y)),$$ and the element in $G(A)^{2n+1}$ associated to $(\delta_{1},\dots, \delta_{n},\delta_{1},\dots, \delta_{n},y)$ is $$(h_{1},\dots, h_{n},\rho(\delta_{1}),\dots, \rho(\prod_{j=1}^{n-1}\delta_{j})^{-1}\rho(\prod_{j=1}^{n}\delta_{j}), \rho(\prod_{j=1}^{n}\delta_{j})^{-1}\rho(\prod_{j=1}^{n}\delta_{j} \cdot y)).$$  We see both $ (\rho(x\prod_{j=1}^{i-1}\delta_{j})^{-1}\rho(x\prod_{j=1}^{i}\delta_{j}))_{i=1,\dots,n}$ and $ (\rho(\prod_{j=1}^{i-1}\delta_{j})^{-1}\rho(\prod_{j=1}^{i}\delta_{j}))_{i=1,\dots,n}$ have a closed orbit and are pre-images of $f(\delta_{1}, \dots, \delta_{n})$, so they are conjugated by some $\phi(x) \in G(A)$.  Since $Z_{G_{A}}(H(\underline{\delta}))$ is minimal by the defining property, $\phi(x)$ must conjugate $\rho(\prod_{j=1}^{n}\delta_{j})^{-1}\rho(\prod_{j=1}^{n}\delta_{j} \cdot y)$ to $\rho(x\prod_{j=1}^{n}\delta_{j})^{-1}\rho(x\prod_{j=1}^{n}\delta_{j} \cdot y)$.  We deduce that $\forall x,y\in \Gamma$, $\rho(x)^{-1}\rho(xy) = \phi(x) \rho(y) \phi(x)^{-1}$, and $\rho$ is a quasi-homomorphism.  It's obvious that for any $(\gamma_{1},\dots, \gamma_{n}) \in \Gamma^n$, $(\rho(\prod_{j=1}^{i-1}\gamma_{j})^{-1}\rho(\prod_{j=1}^{i}\gamma_{j}))_{i=1,\dots, n}$ is a pre-image of $f(\gamma_{1},\dots, \gamma_{n})$.
	\end{proof}

	\subsubsection{Artinian coefficients}\label{artpseudochar}
	
	Let $\bar\rho \colon \Gamma \to G(k)$ be an absolutely $G$-completely reducible representation, and suppose that $H^{0}(\Gamma, \g) = \z$.  We write $\bar f \in \Hom_{\sSet}(B\Gamma, \bar{B}G(k))$ for the map induced from $(\gamma_{1},\dots, \gamma_{n}) \mapsto (\bar\rho(\gamma_{1}), \dots, \bar\rho(\gamma_{n}))$.
	
	\begin{de}
		For $A\in \Art_{\O}$, the set $\aDef_{\bar f}(A)$ is the fiber over $\bar f$ of the map $$\Hom_{\sSet}(B\Gamma, \bar{B}G(A)) \to \Hom_{\sSet}(B\Gamma, \bar{B}G(k)).$$
	\end{de}
	
	\begin{de}
		Let $A\in \Art_{\O}$.  We say a map $\rho\colon \Gamma \to G(A)$ is a quasi-lift of $\bar\rho$ if $\rho \mod \m_{A} = \bar\rho$ and $\rho$ is a quasi-homomorphism.
	\end{de}
	
	\begin{rem}
		In general, a quasi-lift may not be a group homomorphism.  Let $0\to I \to A_{1} \twoheadrightarrow A_{0}$ be an infinitesimal extension in $\Art_{\O}$.  Let $\rho_{0} \colon \Gamma \to G(A_{0})$ be a group homomorphism, let $\sigma\colon G(A_{0}) \to G(A_{1})$ be a set-theoretic section of $G(A_{1}) \to G(A_{0})$ and let $\tilde{\rho} = \sigma \circ \rho_{0}$.  Let's construct a quasi-lift $\rho_{1}=\exp(X_{\alpha}) \tilde{\rho}$ where $X \colon \Gamma \to \g\otimes_{k} I$ is a cochain to be determined.    
		
		For $\alpha, \beta\in \Gamma$, there exists $c_{\alpha,\beta} \in \g \otimes_{k} I$ such that $\tilde{\rho}(\alpha) \tilde{\rho}(\beta) = \exp(c_{\alpha, \beta})\tilde{\rho}(\alpha\beta)$ since $\rho_{0} \colon \Gamma \to G(A_{0})$ is a group homomorphism.  It's easy to check that $c\in Z^{2}(\Gamma, \g \otimes_{k} I)$.  Let $\phi(\alpha) = \exp(Y_{\alpha})$ where $Y\colon \Gamma \to \g \otimes_{k} I$ is a group homomorphism also to be determined.  We require $\rho_{1}(\alpha\beta) =\rho_{1}(\alpha) \phi(\alpha)\rho_{1}(\beta)\phi(\alpha)^{-1}$ for all $\alpha,\beta \in \Gamma$.  Note that $\rho_{1}(\alpha\beta) = \exp(X_{\alpha\beta})\tilde{\rho}(\alpha\beta)$ and 
		\begin{align*}
		\rho_{1}(\alpha) \phi(\alpha)\rho_{1}(\beta)\phi(\alpha)^{-1} &= \exp(X_{\alpha}) \tilde{\rho}(\alpha)\exp(Y_{\alpha})\exp(X_{\beta}) \tilde{\rho}(\beta) \exp(Y_{\alpha})^{-1}\\
		&=\exp(X_{\alpha}) \tilde{\rho}(\alpha) \exp(X_{\beta} +Y_{\alpha} -\Ad\tilde{\rho}(\beta)Y_{\alpha}) \tilde{\rho}(\beta) \\
		&= \exp(X_{\alpha} + \Ad\tilde{\rho}(\alpha)X_{\beta} ) \exp(\Ad \tilde\rho (\alpha) (1-\Ad \tilde\rho (\beta)) Y_{\alpha}) \tilde{\rho}(\alpha)\tilde{\rho}(\beta)\\
		&= \exp(X_{\alpha} + \Ad\tilde{\rho}(\alpha)X_{\beta} ) \exp(\Ad \tilde\rho (\alpha) (1-\Ad \tilde\rho (\beta)) Y_{\alpha}) \exp(c_{\alpha, \beta}) \tilde{\rho}(\alpha\beta)
		\end{align*}
		so we need to find a group homomorphism $Y\colon \Gamma \to \g \otimes_{k} I$ such that $\Ad \tilde\rho (\alpha) (1-\Ad \tilde\rho (\beta)) Y_{\alpha}) + c_{\alpha,\beta}$ is a coboundary.  In particular, in the case $H^{2}(\Gamma, \g) = 0$, we can take an arbitrary group homomorphism $Y\colon \Gamma \to \g$.  Note that $\rho_{1}$ is a group homomorphism if and only if $\phi(\alpha) = \exp(Y_{\alpha}) \in Z(A)$ for any $\alpha \in \Gamma$.  
	\end{rem}
	
	\begin{lem}
		Let $A\in \Art_{\O}$ and let $\rho \colon \Gamma \to G(A)$ be a quasi-lift of $\bar\rho$.  Then $Z(\rho(\Gamma)) = Z(A)$. 
	\end{lem}
	
	\begin{proof}
		See \cite[Lemma 3.1]{Til96} (note that the condition that $\rho$ is a group homomorphism is not used in the proof).
	\end{proof}
	
	\begin{cor}
		Let $A\in \Art_{\O}$ and let $\rho \colon \Gamma \to G(A)$ be a quasi-lift of $\bar\rho$.  Then $\rho$ induces a uniquely determined group homomorphism $\phi\colon\Gamma \to \ker( G^{\ad}(A) \to G^{\ad}(k))$ such that $\rho(x)^{-1}\rho(xy) = \phi(x) \rho(y) \phi(x)^{-1}$ for any $x,y\in \Gamma$.
	\end{cor}
	
	\begin{proof}
		By combining the above lemma with \lemref{quasi-hom}, we see $\phi\colon \Gamma \to G^{\ad}(A)$ is uniquely determined.  Since $\bar\rho$ is a group homomoprhism, $\phi \mod \m_{A}$ commutes with $\bar\rho(\Gamma)$, and hence $\phi \mod \m_{A}$ is trivial.
	\end{proof}
	
	Now we can characterize $\aDef_{\bar f}(A)$ in terms of quasi-lifts.  The following propostion owing to \cite{BHKT19} plays a crucial role (see also its use in the proof of \cite[Theorem 4.10]{BHKT19}):
	
	\begin{pro}\label{BHKT key}
		Suppose that $X$ is an integral affine smooth $\O$-scheme on which $G$ acts.  Let $\underline{x} = (x_{1}, \dots, x_{n}) \in X(k)$ be a point with $G_{k} \cdot x$ closed, and $Z_{G_{k}}(\underline{x})$ scheme-theoretically trivial.  We write $X^{\wedge,\underline{x}}$ for the functor $\Art_{\O} \to \Set$ which sends $A$ to the set of pre-images of $\underline{x}$ under $X(A) \to X(k)$, and write $G^{\wedge}$ for the functor $\Art_{\O} \to \Set$ which sends $A$ to $\ker(G(A) \to G(k))$.  Then 
		\begin{enumerate}
			\item The $G^{\wedge}$-action on $X^{\wedge,\underline{x}}$ is free on $A$-points for any $A\in \Art_{\O}$.
			\item Let $X\dslash G = \Spec \O[X]^{G}$, let $\pi\colon X \to X \dslash G$ be the natural map, and let $(X\dslash G)^{\wedge,\pi(\underline{x})}$ be the functor $\Art_{\O} \to \Set$ which sends $A$ to the set of pre-images of $\pi(\underline{x})$ under $(X\dslash G)(A) \to (X\dslash G)(k)$.  Then $\pi\colon X \to X \dslash G$ induces an isomorphism $X^{\wedge,\underline{x}} /G \cong (X\dslash G)^{\wedge,\pi(\underline{x})}.$
		\end{enumerate}
	\end{pro}
	
	\begin{proof}
		See \cite[Proposition 3.13]{BHKT19}.
	\end{proof}
	
	\begin{cor}
		If $(\gamma_{1},\dots, \gamma_{m})$ is a tuple in $\Gamma^{m}$ such that $(\bar\rho(\gamma_{1}), \dots, \bar\rho(\gamma_{m}))$ has a closed orbit and a scheme-theoretically trivial centralizer in $G_{k}^{\ad}$, then $(\bar\rho(\gamma_{1}), \dots, \bar\rho(\gamma_{m}))$ has a lift $(g_{1}, \dots, g_{m}) \in G(A)^{m}$ which is a pre-image of $f(\gamma_{1},\dots, \gamma_{m})\in \bar{B}G(A)_{m}$, and any other choice is conjugated to this one by a unique element of $G^{\ad}(A)$.
	\end{cor}
	
	\begin{thm}\label{aDef}
		Let $A\in \Art_{\O}$.  Then $\aDef_{\bar f}(A)$ is isomorphic to the set of $\widehat{G}(A)$-conjugacy classes of quasi-lifts of $\bar\rho$.
	\end{thm}
	
	\begin{proof}
		Given a quasi-lift $\rho \colon \Gamma \to G(A)$, then the association $(\gamma_{1},\dots, \gamma_{m}) \mapsto (\rho(\prod_{j=1}^{i-1}\gamma_{j})^{-1}\rho(\prod_{j=1}^{i}\gamma_{j}))_{i=1,\dots,m}$ defines an element of $\aDef_{\bar f}(A)$.
		
		In the following, we will construct a quasi-lift from a given $f \in \aDef_{\bar f}(A)$.
		
		Let $n\geq 1 $ be sufficiently large and choose $\delta_{1},\dots, \delta_{n}\in \Gamma$ such that $(\bar{h}_{1} = \bar\rho(\delta_{1}), \dots, \bar{h}_{n} = \bar\rho(\delta_{n}))$ is a system of generators of $\bar{\rho}(\Gamma)$, then the tuple $(\bar{h}_{1},\dots, \bar{h}_{n})$ has a scheme-theoretically trivial centralizer in $G_{k}^{\ad}$.  By \cite[Corollary 3.7]{BMR05}, the absolutely $G$-completely reducibility implies that the tuple $(\bar{h}_{1},\dots, \bar{h}_{n})$ has a closed orbit.  By the above corollary, we can choose a lift $(h_{1},\dots, h_{n}) \in G(A)^{n}$ of $(\bar{h}_{1},\dots, \bar{h}_{n})$ which is at the same time a pre-image of $f(\delta_{1},\dots, \delta_{n})$.  
		
		For any $\gamma \in \Gamma$, the tuple $(\bar{h}_{1},\dots, \bar{h}_{n}, \bar\rho(\gamma))$ obviously has a closed orbit and a trivial centralizer in $G_{k}^{\ad}$, so we can choose a tuple in $G(A)^{n+1}$ which lifts $(\bar{h}_{1},\dots, \bar{h}_{n}, \bar\rho(\gamma))$ and is a pre-image of $f(\delta_{1},\dots, \delta_{n},\gamma)$.  For this tuple, the first $n$ elements are conjugated to $(h_{1},\dots, h_{n})$ by a unique element of $G^{\ad}(A)$, so there is a unique $g\in G(A)$ such that the tuple is conjugated to $(h_{1}, \dots, h_{n}, g)$.  We define $\rho(\gamma)$ to be this $g$.  It follows immediately that $\rho \mod \m_{A} = \bar{\rho}$.
		
		Now suppose $\gamma_{1},\dots, \gamma_{m}\in \Gamma$.  As above, let $(g_{1},\dots, g_{m})$ be the unique tuple such that $(h_{1},\dots, h_{n}, g_{1},\dots, g_{m})$ lifts $(\bar{h}_{1},\dots, \bar{h}_{n}, \bar\rho(\gamma_{1}), \dots, \bar\rho(\gamma_{m}))$ and is a pre-image of $f(\delta_{1},\dots, \delta_{n},\gamma_{1},\dots, \gamma_{m})$, consider the following diagram, where the horizontal arrows are compositions of face maps:
		\begin{displaymath}
		\xymatrix{
			(\delta_{1}, \dots, \delta_{n},\gamma_{1}, \dots, \gamma_{m})\ar[r]\ar[d]  & (h_{1},\dots, h_{n}, g_{1},\dots, g_{m}) \ar[d] \\
			(\delta_{1}, \dots, \delta_{n},\prod_{j=1}^{i}\gamma_{j}) \ar[r] & (h_{1},\dots, h_{n}, \prod_{j=1}^{i}g_{j})
		}
		\end{displaymath}
		Then $(h_{1},\dots, h_{n}, \prod_{j=1}^{i}g_{j})$ lifts $(\bar{h}_{1},\dots, \bar{h}_{n}, \bar\rho(\prod_{j=1}^{i}\gamma_{j}))$ and is a pre-image of $f(\delta_{1}, \dots, \delta_{n},\prod_{j=1}^{i}\gamma_{j})$.  Hence $\prod_{j=1}^{i} g_{j} = \rho(\prod_{j=1}^{i}\gamma_{j})$, and $g_{i} = \rho(\prod_{j=1}^{i-1}\gamma_{j})^{-1}\rho(\prod_{j=1}^{i}\gamma_{j})$ $(\forall i=1,\dots,m)$.
		
		Let $x,y\in\Gamma$.  Then the element in $G(A)^{2n+2}$ associated to $(\delta_{1},\dots, \delta_{n},x, \delta_{1},\dots, \delta_{n},y)$ is $$(h_{1},\dots, h_{n}, \rho(x), \rho(x)^{-1}\rho(x\delta_{1}),\dots, \rho(x\prod_{j=1}^{n-1}\delta_{j})^{-1}\rho(x\prod_{j=1}^{n}\delta_{j}), \rho(x\prod_{j=1}^{n}\delta_{j})^{-1}\rho(x\prod_{j=1}^{n}\delta_{j} \cdot y)),$$ and the element in $G(A)^{2n+1}$ associated to $(\delta_{1},\dots, \delta_{n},\delta_{1},\dots, \delta_{n},y)$ is $$(h_{1},\dots, h_{n},\rho(\delta_{1}),\dots, \rho(\prod_{j=1}^{n-1}\delta_{j})^{-1}\rho(\prod_{j=1}^{n}\delta_{j}), \rho(\prod_{j=1}^{n}\delta_{j})^{-1}\rho(\prod_{j=1}^{n}\delta_{j} \cdot y)).$$  We see both $ (\rho(x\prod_{j=1}^{i-1}\delta_{j})^{-1}\rho(x\prod_{j=1}^{i}\delta_{j}))_{i=1,\dots,n}$ and $ (\rho(\prod_{j=1}^{i-1}\delta_{j})^{-1}\rho(\prod_{j=1}^{i}\delta_{j}))_{i=1,\dots,n}$ are lifts of $(\bar{h}_{1},\dots, \bar{h}_{n})$ and pre-images of $f(\delta_{1},\dots, \delta_{n})$, so they are conjugated by some $\phi(x) \in G(A)$.  We can even suppose $\phi(x) \in \ker(G(A) \to G(k))$ because the centralizer of $(\bar{h}_{1},\dots, \bar{h}_{n})$ is $Z$.  Since $\phi(x)$ is uniquely determined modulo $Z(A)$, it must conjugate $\rho(\prod_{j=1}^{n}\delta_{j})^{-1}\rho(\prod_{j=1}^{n}\delta_{j} \cdot y)$ to $\rho(x\prod_{j=1}^{n}\delta_{j})^{-1}\rho(x\prod_{j=1}^{n}\delta_{j} \cdot y)$.  We deduce that $\forall x,y\in \Gamma$, $\rho(x)^{-1}\rho(xy) = \phi(x) \rho(y) \phi(x)^{-1}$, and $\rho$ is a quasi-lift.
		
		For the $\rho$ constructed as above, we can recover $f$ from the formula $(\gamma_{1},\dots, \gamma_{m}) \mapsto (\rho(\prod_{j=1}^{i-1}\gamma_{j})^{-1}\rho(\prod_{j=1}^{i}\gamma_{j}))_{i=1,\dots,m}$.
		
		So it remains to prove that if $\rho_{1}$ and $\rho_{2}$ have the same image in $\aDef_{\bar f}(A)$, then they are equal modulo $\ker(G(A) \to G(k))$-conjugation.  Since $(\rho_{1}(\prod_{j=1}^{i-1}\delta_{j})^{-1}\rho_{1}(\prod_{j=1}^{i}\delta_{j}))_{i=1,\dots, n}$ and $(\rho_{2}(\prod_{j=1}^{i-1}\delta_{j})^{-1}\rho_{2}(\prod_{j=1}^{i}\delta_{j}))_{i=1,\dots, n}$ are both lifts of $(\bar{h}_{1},\dots, \bar{h}_{n})$ and pre-images of $f(\delta_{1},\dots, \delta_{n})$, they are conjugated by some $g\in G(A)$, and we may choose $g\in \ker(G(A) \to G(k))$ because the centralizer of $(\bar{h}_{1},\dots, \bar{h}_{n})$ is $Z$.  After conjugation by $g$, we may suppose $(\rho_{1}(\prod_{j=1}^{i-1}\delta_{j})^{-1}\rho_{1}(\prod_{j=1}^{i}\delta_{j}))_{i=1,\dots, n} = (\rho_{2}(\prod_{j=1}^{i-1}\delta_{j})^{-1}\rho_{2}(\prod_{j=1}^{i}\delta_{j}))_{i=1,\dots, n} = (h_{1}^{\prime}, \dots, h_{n}^{\prime})$.  Then for $\gamma\in \Gamma$, $\rho_{k}(\prod_{j=1}^{n}\delta_{j})^{-1}\rho_{k}(\prod_{j=1}^{n}\delta_{j}\cdot \gamma)$ ($k=1,2$) is uniquely determined by the condition: $(h_{1}^{\prime}, \dots, h_{n}^{\prime}, \rho_{k}(\prod_{j=1}^{n}\delta_{j})^{-1}\rho_{k}(\prod_{j=1}^{n}\delta_{j}\cdot \gamma))$ lifts $(\bar{h}_{1},\dots, \bar{h}_{n}, \bar\rho(\gamma))$ and is a pre-image of $f(\delta_{1},\dots, \delta_{n}, \gamma)$.  In consequence, we have $\rho_{1} = \rho_{2}$.
	\end{proof}
	
	For $A\in \Art_{\O}$, let $\aDef_{\bar f,c}(A)$ be the subset of $\aDef_{\bar f}(A)$ consisting of $f\colon B\Gamma \to \bar{B}G(A)$ which factorizes through some finite quotient of $\Gamma$.  In fact we have $\aDef_{\bar f,c}(A) = \Hom_{\sSet _{/ \bar{B}G(k)}}(X, \bar{B}G(A))$ (recall that $X$ is the pro-simplicial set $(B\Gamma_{i})_{i}$).  The following corollary is obvious:
	
	\begin{cor}
		Let $A\in \Art_{\O}$.  Then $\aDef_{\bar f,c}(A)$ is isomorphic to the set of $\widehat{G}$-conjugacy classes of continuous quasi-lifts of $\bar\rho$.
	\end{cor}
	
	As a by-product of the proof of \thmref{aDef}, we also have:
	
	\begin{cor}
		For $A\in \Art_{\O}$, the set $\aDef_{\bar f}(A)$ (resp. $\aDef_{\bar f,c}(A)$) is isomorphic to $\Hom_{\cM}(B\Gamma, BG(A)/G^{\wedge}(A))$ (resp. $\Hom_{\cM}(X, BG(A)/G^{\wedge}(A))$).
	\end{cor}
	
	But unfortunately, the simplicial set $BG(A)/G^{\wedge}(A)$ isn't generally fibrant.
	
	We attempt to compare the difference between $\aDef_{\bar f,c}(A)$ and $\cD(A)$.  Motivated by the front-to-back duality in \cite[8.2.10]{Weib94}, we make the following definition.  Let the reflection action $r$ act on $B\Gamma$ and $\bar{B}G(A)$ as follows:
	\begin{enumerate}
		\item $r$ acts on $B\Gamma_{n} \cong \Gamma \times \dots \times \Gamma$ by $r(\gamma_{1},\dots, \gamma_{n}) = (\gamma_{n},\dots, \gamma_{1})$.
		\item $r$ acts on $\O_{N_{n}G}$ by $r(f)(g_{1},\dots, g_{n}) = f(g_{n},\dots, g_{1})$.  We see that $r$ preserves $\O_{N_{n}G}^{\ad G}$, hence $r$ acts on $\bar{B}G(A)_{n}$.
	\end{enumerate}
	
	\begin{de}
		For $A\in \Art_{\O}$, we define $\bDef_{\bar f}(A)$ (resp. $\bDef_{\bar f,c}(A)$) to be the subset of $\aDef_{\bar f}(A)$ (resp. $\aDef_{\bar f,c}(A)$) consisting of $f \colon B\Gamma \to \bar{B}G(A)$ which commutes with $r$.
	\end{de}
	
	\begin{thm}\label{bDef}
		Let $A\in \Art_{\O}$.  Suppose the characteristic of $k$ is not $2$.  Then $\bDef_{\bar f}(A)$ is in bijection with the set of group homomorphisms $\Gamma \to G(A)$ which lift $\bar\rho$, and $\bDef_{\bar f,c}(A)$ is in bijection with $\cD(A)$.
	\end{thm}
	
	\begin{proof}
		Let $f \in \bDef_{\bar f}(A)$.  It suffices to prove that the quasi-lift $\rho$ obtained in \thmref{aDef} is a group homomorphism.  We choose the tuple $(\delta_{1},\dots, \delta_{n})$ such that $\delta_{i}=\delta_{n+1-i}$ and $\prod_{j=1}^{n}\delta_{j}=e$.  Write $\rho$ for the quasi-lift constructed from this tuple as in \thmref{aDef}, note that the choice of $(\delta_{1},\dots, \delta_{n})$ only affects $\rho$ by some conjugation.  Let $\phi\colon \Gamma \to G(A)/Z(A)$ be the group homomorphism such that $\rho(xy) =\rho(x) \phi(x) \rho(y) \phi(x)^{-1}$ for any $x, y\in \Gamma$.  Note that $\phi(x) \mod \m_{A} = 1$ because $\bar\rho$ is a group homomorphism.  
		
		Since $f$ commutes with $r$, we have 
		\begin{enumerate}
			\item $\rho(x) = \rho(x^{-1})^{-1}$, $\forall x\in \Gamma$.
			\item $\rho(x)^{-1}\rho(xy) = \rho(yx)\rho(x)^{-1}$, $\forall x,y\in \Gamma$.
		\end{enumerate}
		By substituting (1) into $\rho(xy) =\rho(x) \phi(x) \rho(y) \phi(x)^{-1}$, we get $\rho(y^{-1}x^{-1})^{-1} = \rho(x^{-1})^{-1} \phi(x) \rho(y^{-1})^{-1} \phi(x)^{-1}$, then consider $(x,y) \mapsto (x^{-1}, y^{-1})$ and take the inverse, we get $\rho(yx) = \phi(x)^{-1} \rho(y) \phi(x) \rho(x)$.  Now (2) implies $\rho(xy)\rho(x) = \rho(x)\rho(yx)$, which in turn gives $$\rho(x) \phi(x) \rho(y) \phi(x)^{-1} \rho(x) =\rho(x) \phi(x)^{-1} \rho(y) \phi(x) \rho(x).$$  So $\phi(x)^{2}$ commutes with $\rho(\Gamma)$ for any $x\in \Gamma$, and $\phi^{2}=1$.  Since the characteristic of $k$ is not $2$ and $\phi(x) \mod \m_{A} = 1 \in G(k)/Z(k)$, we deduce $\phi= 1$ and $\rho$ is a group homomorphism.
	\end{proof}

	\subsection{Derived deformations of pseudo-characters}\label{derpseudochar}
	
	The functor $\aDef_{\bar f,c} = \Hom_{\sSet _{/ \bar{B}G(k)}}(X, \bar{B}G(-))$ is analogous to the functor $\cD^{\square} = \Hom_{\sSet _{/ BG(k)}}(X, BG(-))$, so it's natural to consider the function complex $\sHom_{\sSet _{/ \bar{B}G(k)}}(X, \bar{B}G(-))$ and then to extend the domain of definition to $\sArt$, as constructing the functor $s\cD\colon \sArt \to \sSet$.
	
	\begin{de}
		For $A\in \sArt$, we define $\barBG(A)$ to be the $\Ex^{\infty}$ of the diagonal of the bisimplicial set $$([p],[q])\mapsto  \Hom_{{}_{\O} \backslash \sCR}(c(\O_{N_{p}G}^{\ad G}),A^{\Delta[q]}),$$ and define $sa\cD(A) = \hofib_{\bar{f}} (\Hom_{\sSet}(X, \barBG(A)) \to \Hom_{\sSet}(X, \barBG(k))).$
	\end{de}
	
	If $A\in \Art_{\O}$, then the bisimplicial set $([p],[q])\mapsto  \Hom_{{}_{\O} \backslash \sCR}(c(\O_{N_{p}G}^{\ad G}),A^{\Delta[q]})$ doesn't depend on the index $q$, and each of its lines is isomorphic to $\Ex^{\infty}\bar{B}G(A)$.  Hence $\bar{f}$ can be regarded as an element of $\Hom_{\sSet}(X, \barBG(k))$.  As the derived deformation functors $s\cD$, we see that $sa\cD \colon \sArt \to \sSet$ is homotopy invariant.  
	
	Note that the inclusion $\O_{N_{\bullet}G}^{\ad G} \hookrightarrow \O_{N_{\bullet}G}$ induces a natural transformation $s\cD \to sa\cD$.  
	
	We would like to understand $\pi_{0} sa\cD(A)$.  Let's first analyse the case $A\in \Art_{\O}$.  For simplicity, we don't take the $\Ex^{\infty}$ here.  Since $BG(A) \to BG(k)$ is a fibration, $\sHom_{\sSet _{/ \bar{B}G(k)}}(X, \bar{B}G(A))$ is a good model for $s\cD(A)$.  However, if $\bar{B}G(A) \to \bar{B}G(k)$ is a not fibration, then $\sHom_{\sSet _{/ \bar{B}G(k)}}(X, \bar{B}G(A))$ is not weakly equivalent to $sa\cD(A)$. 
	
	We have the commutative diagram 
	\begin{displaymath}
	\xymatrix{
		\sHom_{\sSet _{/ \bar{B}G(k)}}(X, BG(A))_{0} \ar[r]\ar[d] & \sHom_{\sSet _{/ \bar{B}G(k)}}(X, \bar{B}G(A))_{0} \ar[d] \\
		\pi_{0}\sHom_{\sSet _{/ BG(k)}}(X, BG(A))  \ar[r] & \pi_{0}\sHom_{\sSet _{/ \bar{B}G(k)}}(X, \bar{B}G(A))}
	\end{displaymath}
	Note that $\pi_{0}sa\cD(A)$ is the coequalizer of $sa\cD(A)_{1} \rightrightarrows sa\cD(A)_{0} = \aDef_{\bar f,c}(A)$ by definition. 
	
	\begin{pro}
		The above diagram is naturally isomorphic to 
		\begin{displaymath}
		\xymatrix{
			\cD^{\square}(A) \ar[r]\ar[d] & \aDef_{\bar f,c}(A) \ar[d] \\
			\cD(A)  \ar[r]\ar@{-->}[ru] & \pi_{0}\sHom_{\sSet _{/ \bar{B}G(k)}}(X, \bar{B}G(A))}
		\end{displaymath}
		And there is a dotted arrow which make the diagram commutative, whose image is $\bDef_{\bar f,c}(A) \subseteq \aDef_{\bar f,c}(A)$.
	\end{pro}
	
	\begin{proof}
		We have $\sHom_{\sSet _{/ BG(k)}}(X, BG(A))_{0} = \Hom_{\cM}(X, BG(A))$, which is exactly $\cD^{\square}(A)$, since $B \colon \Gpd \to \sSet$ is fully faithful.   The other isomorphisms follow by definition.
		
		The dotted arrow signifies the inclusion of usual deformations into pseudo-deformations, whose image is $\bDef_{\bar f,c}(A)$ by \thmref{bDef}.
	\end{proof}
	
	\begin{rem}
		Note however that the functor $sa\cD \colon \sArt \to \sSet$ remains quite mysterious. It may be asked whether there is a more adequate 
		derived deformation functor for pseudo-characters. 
	\end{rem}

	\section{(Co)tangent complexes and pro-representability}\label{cotangent}

	\subsection{Dold-Kan correspondence}
	
	Let's briefly review the Dold-Kan correspondence.  Let $R$ be a commutative ring.  Our goal here is to recall an equivalence (of model categories) between the category of simplicial $R$-modules $\sMod_{R}$ and the category of chain complexes of $R$-modules concentrated on non-negative degrees $\Ch_{\geq 0}(R)$.  Recall the model category structures on $\sMod_{R}$ and $\Ch_{\geq 0}(R)$:
	\begin{enumerate}
		\item For $\sMod_{R}$, the fibrations and weak equivalences are linear morphisms which are in $\sSet$, and the cofibrations are linear morphisms satisfying a lifting property (see \cite[Proposition 7.2.3]{Hir03}).
		\item For $\Ch_{\geq 0}(R)$, the cofibrations, fibrations and weak equivalences are linear morphisms satisfying the following:
		\begin{enumerate}
			\item $f\colon C_{\bullet}\to D_{\bullet}$ is a cofibration if $C_{n}\to D_{n}$ is injective with projective cokernel for $n\geq 0$.
			\item $f\colon C_{\bullet}\to D_{\bullet}$ is a fibration if $C_{n}\to D_{n}$ is surjective for $n\geq 1$.
			\item $f\colon C_{\bullet}\to D_{\bullet}$ is a weak equivalence if the morphism $H_{\ast}f$ induced on homology is an isomorphism.
		\end{enumerate}
	\end{enumerate}
	
	We write $M \in \sMod_{R}$ for the simplicial $R$-module with $M_{n}$ on $n$-th simplicial degree.  Let $N(M)$ be the chain complexes of $R$-modules  such that $N(M)_n=\bigcap\limits_{i=0}^{n-1}\ker(d_{i}) \subseteq M_{n}$ with differential maps
	\begin{displaymath}
	(-1)^{n}d_{n} \colon \bigcap\limits_{i=0}^{n-1}\ker(d_{i})\subseteq M_{n} \to \bigcap\limits_{i=0}^{n-2}\ker(d_{i})\subseteq M_{n-1}.
	\end{displaymath}
	Obviously $M\mapsto N(M)$ is functorial.  We call $N(M)\in\Ch_{\geq 0}(R)$ the normalized complex of $M$.  
	
	The Dold-Kan functor $\DK\colon \Ch_{\geq 0}(R) \to \sMod_{R}$ is the quasi-inverse of $N$.  Explicitely, for a chain of $R$-modules $C_{\bullet}=(C_{0} \leftarrow C_{1} \leftarrow C_{2} \leftarrow \dots)$, we define $\DK(C_{\bullet}) \in \sMod_{R}$ as follows:
	\begin{enumerate}
		\item $\DK(C_{\bullet})_{n} = \bigoplus\limits_{[n] \twoheadrightarrow [k]} C_{k}$.
		\item For $\theta\colon [m] \to [n]$, we define the corresponding $\DK(C_{\bullet})_{n} \to \DK(C_{\bullet})_{m}$ 
		on each component of $\DK(C_{\bullet})_{n}$ indexed by 
		$[n] \stackrel{\sigma}{\twoheadrightarrow} [k]$ as follows: suppose 
		$[m] \stackrel{t}{\twoheadrightarrow} [s] \stackrel{d}{\hookrightarrow} [k]$ 
		is the epi-monic factorization of the composition $[m] \stackrel{\theta}{\rightarrow}  [n] \stackrel{\sigma}{\twoheadrightarrow} [k]$,
		then the map on component $[n] \stackrel{\sigma}{\twoheadrightarrow} [k]$ is 
		\begin{displaymath}
		C_{k} \stackrel{d^{\ast}}{\rightarrow} C_{s} \hookrightarrow \bigoplus\limits_{[m] \twoheadrightarrow [r]} C_{r}.
		\end{displaymath}
	\end{enumerate}
	
	\begin{thm}
		\begin{enumerate}
			\item (Dold-Kan) The functors $\DK$ and $N$ are quasi-inverse hence form an equivalence of categories.  
			Moreover, two morphisms $f,g \in \Hom_{\sMod_{R}}(M, N)$ are simplicially homotopic if and only if $N(f)$ and $N(g)$ are chain homotopic.  
			\item The functors $\DK$ and $N$ preserve the model category stuctures of $\Ch_{\geq 0}(R)$ and $\sMod_{R}$ defined above.
		\end{enumerate}
	\end{thm}
	
	\begin{proof}
		See \cite[Theorem 8.4.1]{Weib94} and \cite[Lemma 2.11]{GJ09}.  Note that (1) is valid for any abelian category instead of $\sMod_{R}$.  
	\end{proof}
	
	\begin{rem} 
		Let $\Ch(R)$ be the category of complexes $(C_i)_{i\in\Z}$ of $R$-modules and $\Ch_{\geq 0}(R)$ the subcategory of complexes for which $C_i=0$ for $i<0$.
		The category $\Ch_{\geq 0}(R)$ is naturally enriched over simplicial $R$-modules, and we have 
		$$\sHom_{\Ch_{\geq 0}(R)}(C_{\bullet}, D_{\bullet}) \cong \sHom_{\sMod_{R}}(\DK(C_{\bullet}), \DK(D_{\bullet})).$$
		Given $C_{\bullet}, D_{\bullet} \in \Ch_{\geq 0}(R)$.  Let $[C_{\bullet}, D_{\bullet}] \in \Ch(R)$ 
		be the mapping complex, more precisely, $[C_{\bullet}, D_{\bullet}]_{n} = \prod_{m} \Hom_{R}(C_{m}, D_{m+n})$ 
		and the differential maps are natural ones.  Let $\tau_{\geq0}$ be the functor which sends a chain complex $X_{\bullet}$ 
		to the truncated complex
		\begin{displaymath}
		0 \leftarrow \ker(X_{0}\to X_{-1}) \leftarrow X_{1} \leftarrow \dots
		\end{displaymath}
		Then there is a weak equivalence
		
		\begin{displaymath}\label{tronc} \sHom_{\Ch_{\geq 0}(R)}(C_{\bullet}, D_{\bullet})\simeq  
		\DK (\tau_{\geq0}[C_{\bullet}, D_{\bullet}])
		\end{displaymath}
		(see \cite[Remark 11.1]{Lur09}). 
		
		It's clear that $\pi_{i}\sHom_{\Ch_{\geq 0}(R)}(C_{\bullet}, D_{\bullet})$ is isomorphic 
		to the chain homotopy classes of maps from $C_{\bullet}$ to $D_{\bullet +n}$.
	\end{rem}

	\subsection{(Co)tangent complexes of simplicial commutative rings}
	
	We recall Quillen's cotangent and tangent complexes of simplicial commutative rings.  
	
	Let $A$ be a commutative ring.  For $R$ an $A$-algebra, let $\Omega_{R/A}$ be the module of differentials with the canonical $R$-derivation $d\colon R \to \Omega_{R/A}$.  Let $\Der_{A}(R, -)$ be the covariant functor which sends an $R$-module $M$ to the $R$-module 
	\begin{displaymath}
	\Der_{A}(R, M) = \{ D\colon R \to M \mid D \text{ is } A\text{-linear and } D(xy) = xD(y)+yD(x),\text{ } \forall x,y\in R \}.
	\end{displaymath}
	It's well-known that $\Hom_{R}(\Omega_{R/A}, -)$ is naturally isomorphic to $\Der_{A}(R, -)$ via $\phi \mapsto \phi \circ d$.  
	
	Let $T$ be an $A$-algebra, and let ${}_{A}\backslash \CR/_{T}$ be the category of commutative rings $R$ over $T$ and under $A$.  
	Then for any $T$-module $M$ and any $R \in {}_{A}\backslash \CR/_{T}$, we have natural isomorphisms
	\begin{displaymath}
	\Hom_{T}(\Omega_{R/A} \otimes_{R} T, M) \cong \Der_{A}(R, M)\cong \Hom_{{}_{A}\backslash \CR/_{T}}(R, T \oplus M).
	\end{displaymath}
	where $T\oplus M$ is the $T$-algebra with square-zero ideal $M$. So the functor $R \mapsto \Omega_{R/A} \otimes_{R} T$ is left ajoint to the functor $M \mapsto T \oplus M$.  
	
	The above isomorphisms have level-wise extensions to simplicial categories (see \cite{GJ09} 
	Lemma \RNum{2}.2.9 and Example \RNum{2}.2.10). 
	For $R \in {}_{A}\backslash \sCR$, we can form $\Omega_{R/A}\otimes_R T\in s\Mod_T$.
	
	We have 
	\begin{displaymath}
	\sHom_{\sMod_{T}}(\Omega_{R/A} \otimes_{R} T, M) \cong \sHom_{{}_{A}\backslash \sCR/_{T}}(R, T \oplus M).
	\end{displaymath}
	The functor $M \mapsto T \oplus M$ from $\sMod_{T}$ to ${}_{A}\backslash \sCR/_{T}$ 
	preserves fibrations and weak equivalences (we may see this via the Dold-Kan correspondence), 
	so the left adjoint functor $R \mapsto \Omega_{R/A} \otimes_{R} T$ is left Quillen and it admits a total left derived functor.  We introduce the cotangent complex $L_{R/A}$ in the following definition, 
	so that the total left derived functor has the form $R \mapsto L_{R/A} \underline{\otimes}_{R} T$.
	Note that given two simplicial modules $M,N$ over a simplicial ring $S$, one can form (degreee-wise) a tensor product, denoted $M\underline{\otimes}_S N$, which is a simplicial $S$-module.	
	\begin{de}
		For $R \in {}_{A}\backslash \sCR$, we define $L_{R/A} = \Omega_{c(R)/A} \underline{\otimes}_{c(R)} R \in \sMod_{R}$, where $c(R)$ 
		is the middle object of some cofibration-trivial fibration factorization $A \hookrightarrow c(R) \stackrel{\sim}{\twoheadrightarrow} R$, 
		and we call $L_{R/A}$ the cotangent complex of $R$.
	\end{de}
	Note that it is an abuse of language, as it should be called cotangent simplicial $R$-module, 
	because for $R$ simplicial, $L_{R/A}\in \sMod_R$ but there is no notion of complexes of $R$-modules.
	
	By construction, $L_{R/A} \underline{\otimes}_{R} T$ is cofibrant as it's the image of the cofibrant object $c(R)$ 
	under a total left derived functor, and it is fibrant in $\sMod_{R}$ (all objects are fibrant there).  
	Note also that the weak equivalence class of $L_{R/A} \underline{\otimes}_{R} T$ is independent of the choice of $c(R)$. 
	It follows from these two observations that $L_{R/A}$ is determined up to homotopy equivalence 
	(by the Whitehead theorem \cite[Theorem 7.5.10]{Hir03}). 
	Using the Dold-Kan equivalence, we can form the normalized complex (determined up to homotopy equivalence)
	$$N( L_{R/A} \underline{\otimes}_R T)\in \Ch_{\geq 0}(T).$$
	From now on, we keep the functor $N$ understood and simply write
	$$L_{R/A}  \otimes_R T\in \Ch_{\geq 0}(T)$$	
	Recall that for $M,N\in\Ch(T)$, the internal Hom $[M,N]\in\Ch(T)$ is defined as 
	$$[M,N]_n=\prod_{m} \Hom_{T}(M_{m}, N_{m+n}).$$
	Note that if $M\in\Ch_{\geq 0}(T)$, then $[M,T]\in \Ch_{\leq 0}(T)$. For $C\in\Ch_{\leq 0}(T)$, we
	write $C^i=C_{-i}$ for $i\geq 0$; we thus identify $\Ch_{\leq 0}(T)=\Ch^{\geq 0}(T)$.
	
	Since we'll consider internal Homs $[L_{R/A} \otimes_R T, M]$, for (classical or simplicial) $T$-modules $M$,  we define 
	
	\begin{de}
		The $T$-tangent complex $\gt R_T$ of $R\to T$ is the internal hom complex 
		$$[L_{R/A} \otimes_R T, T] \in \Ch^{\geq 0}(T).$$
	\end{de}
	
	Then $\gt R_T$ is well-defined up to chain homotopy equivalence since it is the case for $L_{R/A}  \otimes_R T$. 
	
	For $R \in {}_{A}\backslash \sCR/_{T}$ and $C_{\bullet} \in \Ch_{\geq 0}(T)$, we have (by \remref{tronc}):
	\begin{align*}
	\sHom_{ {}_{A} \backslash \sCR /_{T}}(c(R), T\oplus \DK(C_{\bullet})) 
	&\cong \sHom_{\sMod_{T}}(L_{R/A} \otimes_{R} T,\DK(C_{\bullet})) 
	\\ &\simeq  \DK(\tau_{\geq0}[L_{R/A} \otimes_{R} T,C_{\bullet}])
	\\ &\cong \DK(\tau_{\geq0} [L_{R/A},C_{\bullet}]).
	\end{align*}

	\subsection{Tangent complexes of formally cohesive functors and Lurie's criterion}
	
	In \cite[Section 4]{GV18}, the authors define the tangent complexes of formally cohesive functors.  To summarize, we have the following proposition:
	
	\begin{pro}
		Let $\cF \colon \sArt \to \sSet$ be a formally cohesive functor.  Then there exists $L_{\cF} \in \Ch(k)$ such that $\cF(k\oplus \DK(C_{\bullet}))$ is weakly equivalent to $\DK(\tau_{\geq0} [L_{\cF},C_{\bullet}])$ for every $C_{\bullet} \in \Ch_{\geq 0}(k)$ with $H_{\ast}(C_{\bullet})$ finite over $k$.
	\end{pro}
	
	\begin{proof}
		See \cite[Lemma 4.25]{GV18}.
	\end{proof}
	
	\begin{de}
		Let $\cF \colon \sArt \to \sSet$ be a formally cohesive functor.
		\begin{enumerate}
			\item We call $L_{\cF}$ the cotangent complex of $\cF$.
			\item The tangent complex $\gt\cF$ of $\cF$ is the chain complex defined by the internal hom complex $[L_{\cF}, k]$. 
		\end{enumerate}
	\end{de}
	
	Note that $L_{\cF}$ and $\gt\cF$ are uniquely determined up to quasi-isomorphism.  We shall use $\gt^{i}\cF$ to abbreviate the homology groups $H_{-i}\gt\cF$.
	
	\begin{rem}
		If $R \in {}_{\cO}\backslash \sCR/_{k}$ is cofibrant, then the functor $\cF_{R} = \sHom_{{}_{\O} \backslash \sCR/_{k}}(R, -) \colon \sArt \to \sSet$ is formally cohesive.  Since $\DK(\tau_{\geq0} [L_{\cF_{R}},k[n]]) \simeq \sHom_{ {}_{\O} \backslash \sCR /_{k}}(R, k\oplus k[n]) \simeq \DK(\tau_{\geq0} [L_{R/\O},k[n]])$, the cotangent complexes $L_{\cF_{R}}$ and $L_{R/\O} \otimes_{R} k$ are quasi-isomorphic.
	\end{rem}
	
	\begin{de}
		We say a functor $\cF\colon \sArt \to \sSet$ is pro-representable, if there exists a projective system $R = (R_{n})_{n\in \N}$ with each $R_{n} \in \sArt$ cofibrant, such that $\cF$ is weakly equivalent to $\varinjlim\limits_{n}\sHom_{\sArt}(R_{n}, -)$.  In this case we say $R = (R_{n})_{n\in \N}$ is a representing ring for $\cF$.  We shall write $\sHom_{\sArt}(R, -)$ for $\varinjlim\limits_{n}\sHom_{\sArt}(R_{n}, -)$.
	\end{de}
	
	\begin{rem}
		The pro-representability defined above is called sequential pro-representability in \cite{GV18}, but we will only deal with this case.
	\end{rem}
	
	\begin{thm}[Lurie's criterion]\label{Lurie}
		Let $\cF$ be a formally cohesive functor.  If $\dim_{k} \gt^{i}\cF$ is finite for every $i \in \Z$, and $\gt^{i}\cF = 0$ for every $i <0$, then $\cF$ is (sequentially) pro-representable.
	\end{thm}	
	
	\begin{proof}
		See \cite[Corollary 6.2.14]{Lur04} and \cite[Theorem 4.33]{GV18}.
	\end{proof}
	
	The following lemma illustrates the conservativity of the tangent complex functor:
	
	\begin{lem}
		Suppose $\cF_{1},\cF_{2}\colon \sArt \to \sSet$ are formally cohesive functors.  Then a natural transformation $\cF_{1} \to \cF_{2}$ is a weak equivalence if and only if it induces isomorphisms $\gt^{i}\cF_{1} \to \gt^{i}\cF_{2}$ for all $i$.
	\end{lem}
	
	\begin{proof}
		If the natural transformation induces isomorphisms $\gt^{i}\cF_{1} \to \gt^{i}\cF_{2}$, then $ \cF_{1}(k\oplus k[n]) \to \cF_{2}(k\oplus k[n])$ is a weak equivalence.  So by simplicial artinian induction \cite[Section 4]{GV18}, it induces a weak equivalence $\cF_{1}(A) \to \cF_{2}(A)$ for $A \in \sArt$.
	\end{proof}

	\subsection{Pro-representability of derived deformation functors}
	
	In the following, we suppose $p>2$, and $\Gamma=G_{F,S}$ for $S=S_{p} \cup S_{\infty}$.  Suppose further that $\bar\rho$ satisfies $(Ord_v)$ and $(Reg_v)$ for $v\in S_{p}$, and $H^0(\Gamma, \g_k) = \z_k$.  Recall that we've 
	introduced derived deformation functors $s\cD$ and $s\cD^{\ord}$, 
	as well as the modifying-center variants $s\cD_{Z}$ and $s\cD_{Z}^{\ord}$.  
	These functors are all formally cohesive.  Their tangent complexes are related to the Galois cohomology groups $H^i_\ast(\Gamma,\g_k)$ of adjoint representations, where $\ast=\emptyset$ or $\ord$.

	\subsubsection{Galois cohomology}
	
	We briefly review the Galois cohomology theory.  To define the nearly ordinary cohomology, we fix
	the standard Levi decomposition $B=TN$ of the standard Borel of $G$; it induces a decomposition of Lie algebras over $k$: 
	$\b_k=\t_k\oplus\n_k$. Recall the definition of the Greenberg-Wiles nearly ordinary Selmer group
	$$\widetilde{H}_{\ord}^{1}(\Gamma, \g_k)=\Ker\left(H^{1}(\Gamma, \g_k)\to\prod_{v\in S_p}{H^1(\Gamma_v,\g_k)\over L_v}\right)$$
	where $L_v=\im(H^1(\Gamma_v,\b_k)\to H^1(\Gamma_v,\g_k))$.
	
	For $v\in S_{p}$, let $\tilde{L}_{v} \subseteq Z^1(\Gamma_{v}, \g_k)$ be the preimage of $L_v$.  Let $C^{\bullet}_{\ord}(\Gamma, \g_k)$ be the mapping cone of the natural cochain morphism 
	\begin{displaymath}
	\xymatrix{
		0 \ar[r]& C^{0}(\Gamma, \g_k) \ar[r]\ar[d] &  C^{1}(\Gamma, \g_k) \ar[r]\ar[d] &  C^{2}(\Gamma, \g_k) \ar[d]\ar[r] &\dots \\
		0 \ar[r] & 0 \ar[r] & \bigoplus_{v \in S_p} C^{1}(\Gamma_{v}, \g_k)/\tilde{L}_{v}  \ar[r] & \bigoplus_{v \in S_p} C^{2}(\Gamma_{v}, \g_k) \ar[r] &\dots 
	}
	\end{displaymath}
	Then we define the nearly ordinary cohomology groups $H^{\ast}_{\ord}(\Gamma, \g_k)$ as the cohomology of the complex  $C^{\bullet}_{\ord}(\Gamma, \g_k)$.  They fit into the exact sequence ($\bigstar$):
	\begin{align*}
	0 \to &H^{0}_{\ord}(\Gamma, \g_k) \to H^{0}(\Gamma, \g_k) \to 0 \\
	\to &H^{1}_{\ord}(\Gamma, \g_k) \to H^{1}(\Gamma, \g_k) \to \bigoplus_{v \in S_p} H^{1}(\Gamma_v, \g_k)/L_{v} \\
	\to &H^{2}_{\ord}(\Gamma, \g_k) \to H^{2}(\Gamma, \g_k) \to \bigoplus_{v \in S_p} H^{1}(\Gamma_v, \g_k)\\
	\to &H^{3}_{\ord}(\Gamma, \g_k) \to 0
	\end{align*}
	In particular, $\widetilde{H}_{\ord}^{1}(\Gamma, \g_k)=H_{\ord}^{1}(\Gamma, \g_k)$.
	
	\begin{de} \label{dual} For a finite $\cO[\Gamma]$-module $M$, we write $M^\vee=\Hom_\O(M,K/\O)$ and $M^\ast=\Hom_\O(M,K/\O(1))$. In particular if $M$ is a $k$-vector space,
	$M^\vee=\Hom_k(M,k)$ and $M^\ast=\Hom_k(M,k(1))$.
	\end{de}
	
	Recall the local Tate duality $H^1(\Gamma_v, \g_k) \times H^1(\Gamma_v, \g_k^{\ast}) \to k$.  Let $L_v^{\perp} \subseteq H^1(\Gamma_v, \g_k^{\ast})$ be the dual of $L_v$.  We define similarly the cohomology groups $H^{\ast}_{\ord,{}^\perp}(\Gamma,\g_k^{\ast})$.  In particular  
	$$\bigoplus_{v \in S_p}L_{v} \to H^{1}(\Gamma, \g_k^{\ast})^{\vee} \to H^{1}_{\ord,{}^\perp}(\Gamma, \g_k^{\ast})^{\vee} \to 0$$
	is exact.  By fitting this into the Poitou-Tate exact sequence (see \cite[Theorem \RNum{1}.4.10]{Mil06}), we obtain the exact sequence ($\bigstar\bigstar$): 
	\begin{align*}
	&H^{1}(\Gamma, \g_k) \to \bigoplus_{v \in S_p} H^{1}(\Gamma_v, \g_k)/L_{v}  \\
	\to H^{1}_{\ord,{}^\perp}(\Gamma, \g_k^{\ast})^{\vee} \to &H^{2}(\Gamma, \g_k) \to \bigoplus_{v \in S_p} H^{2}(\Gamma_v, \g_k) \\
	\to H^{0}(\Gamma, \g_k^{\ast})^{\vee} \to &0
	\end{align*} 
	We deduce the Poitou-Tate duality: 
	
	\begin{thm}\label{Poitou-Tate}
		For each $i\in \{0,1,2,3\}$, there is a perfect pairing 
		\begin{displaymath}
		H^{i}_{\ord,{}^\perp}(\Gamma,\g_k^{\ast}) \times H^{3-i}_{\ord}(\Gamma, \g_k) \to k.
		\end{displaymath}
	\end{thm}
	
	\begin{proof}
		For $i\in \{0,1\}$, it suffices to compare the exact sequences ($\bigstar$) and ($\bigstar\bigstar$).  The cases $i\in \{2,3\}$ follow by duality.
	\end{proof}

	\subsubsection{Tangent complex}
	
	\begin{lem}\label{tangent complex calculation}
		\begin{enumerate}
			\item We have $\gt^{i} s\cD \cong H^{i+1}(\Gamma, \g_k)$  for all $i \in \Z$.  On the other hand, $\gt^{i} s\cD_{Z} \cong \gt^{i} s\cD$ when $i \neq -1$, and $\gt^{-1} s\cD_{Z}=0$.
			\item Let $v\in S_{p}$.  Then we have $\gt^{i} s\cD_v \cong H^{i+1}(\Gamma_v, \g_k)$  for all $i \in \Z$.  On the other hand, $\gt^{i} s\cD_{v,Z} \cong \gt^{i} s\cD_v$ when $i \neq -1$, and $\gt^{-1} s\cD_{v,Z} \cong H^{0}(\Gamma_v, \g_k)/ \z_k$.
			\item Let $v\in S_{p}$.  Then we have $\gt^{i} s\cD^{\ord}_v \cong H^{i+1}(\Gamma_v, \b_k)$  for all $i \in \Z$.  On the other hand, $\gt^{i} s\cD^{\ord}_{v,Z} \cong \gt^{i} s\cD^{\ord}_v$ when $i \neq -1$, and $\gt^{-1} s\cD^{\ord}_{v,Z} \cong H^{0}(\Gamma_v, \b_k)/ \z_k$.  Moreover, $\gt^{1} s\cD_{v}^{\ord}= 0$ if $(Reg_{v}^{\ast})$ holds.
		\end{enumerate}
	\end{lem}
	
	\begin{proof}
		Note that $\gt^{j-i} \cF \cong \pi_{i} \cF(k \oplus k[j])$ for any formally cohesive functor $\cF$ 
		and any $i, j\geq0$.  Later in section \ref{relative} we shall give a slightly generalized version of the lemma.  
		See also \cite[Section 7.3]{GV18}.
	\end{proof}	
	
	In particular, by Lurie's criterion (\thmref{Lurie}), this lemma together with the finiteness of the cohomology groups, implies 
	
	\begin{cor} 
		The center-modified functor $s\cD_{Z}$ is pro-representable.
	\end{cor}
	
	Now we treat the nearly ordinary case $s\cD^{\ord}_{Z}$.  Let's recall that $s\cD_{\loc,Z}= \prod_{v\in S_{p}} s\cD_{v,Z}$, $s\cD_{\loc,Z}^{\ord}= \prod_{v\in S_{p}} s\cD_{v,Z}^{\ord}$, and $s\cD^{\ord}_{Z} = s\cD_{Z} \times^{h}_{s\cD_{\loc,Z}} s\cD_{\loc,Z}^{\ord}$.  
	Recall that $\bar\rho$ satisfies $(Ord_v)$ and $(Reg_v)$ for $v\in S_{p}$, so $s\cD^{\ord}_{Z}$ is indeed the derived generalization of $\cD^{\ord}$, i.e., 
	$\pi_{0} s\cD^{\ord}_{Z}(A) \cong \cD^{\ord}(\pi_{0} A)$ for homotopy discrete $A\in \sArt$ (see \propref{simplord}).
	
	\begin{lem}
		Suppose furthermore $(Reg_{v}^{\ast})$ for $v\in S_p$.  Then $\gt^{i} s\cD^{\ord}_{Z} \cong H^{i+1}_{\ord}(\Gamma, \g)$ when $i \geq 0$, and $\gt^{i} s\cD^{\ord}_{Z} =0$ when $i<0$.
	\end{lem}
	
	\begin{proof}
		We have the Mayer-Vietoris exact sequence (see \cite[Lemma 4.30 (iv)]{GV18} and \cite[Section 1.5]{Weib94}) 
		\begin{align*}
		\gt^{i} s\cD^{\ord}_{Z} \to \gt^{i} s\cD_{Z} \oplus \gt^{i} s\cD_{\loc,Z}^{\ord} \to \gt^{i} s\cD_{\loc,Z} \stackrel{[1]}{\to} \dots.
		\end{align*}
		By \lemref{tangent complex calculation}, we obtain an exact sequence 
		\begin{align*}
		0 &\to \gt^{-1} s\cD^{\ord}_{Z} \to \bigoplus_{v \in S_p} H^{0}(\Gamma_v, \b_k)/ \z_k \to \bigoplus_{v \in S_p} H^{0}(\Gamma_v, \g_k)/ \z_k \\ &\to \gt^{0} s\cD^{\ord}_{Z} \to H^{1}(\Gamma, \g_k) \oplus (\bigoplus_{v \in S_p} H^{1}(\Gamma_{v}, \b_k)) \to \bigoplus_{v \in S_p} H^{1}(\Gamma_{v}, \g_k) \\ &\to \gt^{1} s\cD^{\ord}_{Z} \to H^{2}(\Gamma, \g_k) \to \bigoplus_{v \in S_p} H^{2}(\Gamma_{v}, \b_k)  \\ &\to \gt^{2} s\cD^{\ord}_{Z} \to 0
		\end{align*}
		By assumtion $(Reg_v)$, the map $H^{0}(\Gamma_v, \b_k)/ \z_k \to H^{0}(\Gamma_v, \g_k)/ \z_k$ is an isomorphism.  The conclusion follows from comparing the above exact sequence with ($\bigstar$).
	\end{proof}
	
	In particular $\gt^{-1} s\cD^{\ord}_{Z}=0$ (note that for this we don't need $(Reg_{v}^{\ast})$).  By Lurie's criterion (\thmref{Lurie}) and the finiteness of the cohomology groups, we have the following corollary: 
	
	\begin{cor} 
		The functor $s\cD^{\ord}_{Z}$ is pro-representable.
	\end{cor}
	
	Let $R^{s,\ord}$ be a representing (pro-)simplicial ring.  Since $\pi_{0} s\cD^{\ord}_{Z}(A) \cong \cD^{\ord}(A)$ for $A \in \Art_{\O}$, the ring $\pi_{0}R^{s,\ord}$ represents the classical nearly ordinary deformation functor $\cD^{\ord}$.

	\subsection{Relative derived deformations and relative tangent complexes}\label{relative}
	
	Let $T \in \Art_{\O}$ and let $\rho_{T}\colon \Gamma \to G(T)$ be a nearly ordinary representation. 
	For $v\in S_{p}$, we write $\rho_{T,v}$ for the restriction of $\rho_{T}$ on $\Gamma_{v}$ and 
	we suppose the image of $\rho_{T,v}$ lies in $B(T)$ (more precisely, we should say the image of some conjugation of $\rho_{T,v}$ lies in $B(T)$, but there is no crucial difference).  
	Let $X$ and $X_v$ be the pro-simplicial sets associated to the profinite groups $\Gamma$ and $\Gamma_v$.
	We identify $\rho_{T}$ as a map of (pro-)simplicial sets $X \to BG(T) \to \BG(T)$ (here $BG(T)$ is 
	the classical classifying space of the finite group $G(T)$ and $\BG(T)$ is a fibrant replacement, see \defref{BG}) and identify $\rho_{T,v}$ 
	as $X_{v} \to BB(T) \to \BB(T) \to \BG(T)$.
	
	Let's consider the derived deformations functors over $\rho_{T}$.
	
	\begin{de}
		\begin{enumerate}
			\item Let $s\cD_{\rho_{T}} \colon\sArtT \to \sSet$ be the functor 
			$$A \mapsto \hofib_{\rho_{T}} ( \sHom_{\sSet}(X,\BG(A)) \to \sHom_{\sSet}(X,\BG(T))).$$
			\item For $v\in S_{p}$, let $s\cD_{\rho_{T,v}} \colon \sArtT \to \sSet$ 
			be the functor $$A \mapsto \hofib_{\rho_{T,v}} ( \sHom_{\sSet}(X_{v},\BG(A)) \to \sHom_{\sSet}(X_{v},\BG(T))).$$
			\item For $v\in S_{p}$, let $s\cD^{\ord}_{\rho_{T,v}} \colon \sArtT \to \sSet$ 
			be the functor $$A \mapsto \hofib_{\rho_{T,v}} ( \sHom_{\sSet}(X_{v},\BB(A)) \to \sHom_{\sSet}(X_{v},\BB(T))).$$
		\end{enumerate}
	\end{de}
	
	Our goal is to prove the following proposition (see also \cite[Example 4.38 and Lemma 5.10]{GV18}):
	
	\begin{pro}\label{main}
		Let $M$ be a finite module over an arbitrary Artin ring $T$.  Then for $i, j \geq 0$ we have $$\pi_{i} s\cD_{\rho_{T}}(T\oplus M[j]) \cong H^{1+j-i} (\Gamma, \g_{T}\otimes_{T} M).$$ 
	\end{pro}
	
	Note that $\sHom_{\sSet}(X,-)$ is defined by the filtered colimit $\varinjlim_{i}\sHom_{\sSet}(B\Gamma_{i},-)$, which commutes with homotopy pullbacks.  	So it suffices to prove the proposition with $\Gamma$ replaced by $\Gamma_i$ and $X$ replaced by $B\Gamma_{i}$.  To simplify the notations, we suppose $\Gamma$ is a finite group during the proof.
	
	\begin{lem}\label{calculation}
		Let $A\in \sArtT$.  Then $s\cD_{\rho_{T}}(A)$ is weakly equivalent to 
		$$\holim_{\bm{\Delta}X} \hofib_{\ast}(\BG(A) \to \BG(T)).$$
	\end{lem}
	
	\begin{proof}
		By \cite[Proposition 18.9.2]{Hir03}, $X$ is weakly equivalent to $\hocolim_{(\bm{\Delta}X)^\op} \ast$ (i.e., the homotopy colimit of the single-point simplicial set indexed by $(\bm{\Delta}X)^\op$).  Hence (see \cite[Theorem 18.1.10]{Hir03}) $$\sHom_{\sSet}(X, \BG(A)) \simeq \holim_{\bm{\Delta}X} \sHom_{\sSet}(\ast, \BG(A)) \simeq \holim_{\bm{\Delta}X}\BG(A),$$ and $$\sHom_{\sSet}(X, \BG(T)) \simeq \holim_{\bm{\Delta}X} \sHom_{\sSet}(\ast, \BG(T)) \simeq \holim_{\bm{\Delta}X}\BG(T).$$
		Note that $\rho_{T}$, as the single-point simplicial subset of $\sHom_{\sSet}(X, \BG(T))$, is identified with $\holim_{\bm{\Delta}X}\ast \to \holim_{\bm{\Delta}X}\BG(T)$.  Since homotopy limits commute with homotopy pullbacks, we conclude that 
		$$s\cD_{\rho_{T}}(A) \simeq \holim_{\bm{\Delta}X} \hofib_{\ast}(\BG(A) \to \BG(T)).$$
	\end{proof}
	
	Let's first analyse $\hofib_{\ast}(\BG(A) \to \BG(T))$.
	
	\begin{lem}
		The homotopy groups of $\hofib_{\ast}(\BG(A) \to \BG(T))$ are trivial except at degree $j+1$, where it is $\g_{T}\otimes_{T} M$.
	\end{lem}
	
	\begin{proof}
		Note that $A\mapsto \hofib_{\ast}(\BG(A) \to \BG(T))$ preserves weak equivalences and homotopy pullbacks.
		
		Since $T\oplus M[j] \to T$ is $j$-connected, the map $\BG(T\oplus M[j]) \to \BG(T)$ is $(j+1)$-connected (see \cite[Corollary 5.3]{GV18}), and the homotopy groups of the homotopy fiber vanish up to degree $j$.  Since the functor $A\mapsto \hofib_{\ast}(\BG(A) \to \BG(T))$ maps the homotopy pullback square
		\begin{displaymath}
		\xymatrix{
			T\oplus M[j-1] \ar[r]\ar[d] & T \ar[d] \\
			T \ar[r] & T \oplus M[j]}
		\end{displaymath}
		to a homotopy pullback square, we get 
		$$\pi_{j+k} \hofib_{\ast}(\BG(T \oplus M[j]) \to \BG(T)) \cong \pi_{j+k-1} \hofib_{\ast}(\BG(T \oplus M[j-1]) \to \BG(T))$$  
		for any $k\geq 0$.  Consequently 
		$$\pi_{j+k} \hofib_{\ast}(\BG(T \oplus M[j]) \to \BG(T)) \cong \pi_{k} \hofib_{\ast}(\BG(T \oplus M[0]) \to \BG(T)),$$
		and $\hofib_{\ast}(\BG(T \oplus M[j]) \to \BG(T))$ has homotopy groups concentrated on degree $j+1$, where it is $\g_{T}\otimes_{T} M$.  
	\end{proof}
	
	Let $Y$ be the $\bm{\Delta}X$-diagram in $\sSet$ (i.e, functor $\bm{\Delta}X \to \sSet$) which takes the value $\hofib_{\ast}(\BG(A) \to \BG(T))$.  Then $Y$ is a local system (see \cite[Definition 4.34]{GV18}, it's called the cohomological coefficient system in \cite[Page 28]{GM13}) on $X$.  There is hence a $\pi_{1} (X,\ast)$-action on the homotopy group $\g_{T}\otimes_{T} M$.  By unwinding the constructions, we see this is the conjugacy action of $\rho_{T}$ on $\g_{T}\otimes_{T} M$.
	
	It suffices to calculate $\holim Y$.  Under the Dold-Kan correspondence, we may identify $\hofib_{\ast}(\BG(A) \to \BG(T))$ with the chain complex with homology $\g_{T}\otimes_{T} M$ concentrated on degree $j+1$.  But in fact it's more convenient to regard $\hofib_{\ast}(\BG(A) \to \BG(T))$ as a cochain complex with cohomology $\g_{T}\otimes_{T} M$ concentrated on degree $-(j+1)$, because the homotopy limit of cochain complexes is drastically simple (see \cite[Section 19.8]{Dug08}).  By shifting degrees, it suffices to suppose that the cohomology is concentrated on degree $0$.
	
	\begin{lem}\label{calculation 2}
		Let $N$ be a $T[\Gamma]$-module, and we regard $N$ as a cochain complex concentrated on degree $0$.  Let $Y$ be the $\bm{\Delta}X$-diagram in $\Ch^{\geq 0}(T)$ (i.e, functor $\bm{\Delta}X \to \Ch^{\geq 0}(T)$) which takes the value $N$.  Then $\holim Y \simeq C^{\bullet}(\Gamma, N).$  Here $C^{\bullet}(\Gamma, N)$ is the cochain which computes the usual group cohomology.
	\end{lem}
	
	\begin{proof}
		By \cite[Lemma 18.9.1]{Hir03}, $\holim Y$ is naturally isomorphic to the homotopy limit of the cosimplicial object $Z$ in $\Ch^{\geq 0}(T)$ whose codegree $[n]$ term is $\prod_{\sigma \in X_{n}} Y_{\sigma} = \prod_{\sigma \in X_{n}} N$.  We have to explain the coface maps of $Z$.  For this purpose we describe $Z=(Z^n)_n$ as follows:
		
		The $T[\Gamma]$-module $N$ defines a functor $D$ from the one-object groupoid $\bullet$ with $\End(\bullet)=\Gamma$ to $\Ch^{\geq 0}(T)$, such that $D(\bullet) = N$, and $D(\Gamma)$ acts on $N$ by the $\Gamma$-action.  Then $Z^n$ is $\prod\limits_{i_0 \to \dots \to i_n} D(i_n)$ (all $i_k$'s are equal to the object $\bullet$ here, but keeping the difference helps to clarify the process).  Let $d_k$ be the $k$-th face map from $\Gamma^{n+1}$ to $\Gamma^n$, in other words, $d_k$ maps $(i_0 \to \dots \to i_{n+1})$ to $(j_0 \to \dots \to j_{n})$ by "covering up" $i_k$.  Then the corresponding $D(j_n) \to D(i_{n+1})$ is the identity map if $k\neq n+1$, and is $D(i_n \to i_{n+1})$ if $k=n+1$.  
		
		By \cite[Proposition 19.10]{Dug08}, $\holim Z$ is quasi-isomorphic to the total complex of the alternating double complex defined by $Z$.  Since each $Z^{n}$ is concentrated on degree $0$, the total complex is simply 
		$$\dots \to \prod\limits_{\Gamma^n} N \to \prod\limits_{\Gamma^{n+1}} N \to \dots$$
		and the alternating sum $\prod\limits_{\Gamma^{n}} N \to \prod\limits_{\Gamma^{n+1}} N$ is exactly the one which occurs in computing group cohomology.  We conclude that $\holim Y \simeq \holim Z \simeq C^{\bullet}(\Gamma, N)$.  
	\end{proof}
	
	Now we can prove \propref{main}:	
	
	\begin{proof}
		From the above discussions, $s\cD_{\rho_{T}}(T\oplus M[j])$ corresponds to $\tau^{\leq0}C^{\bullet+j+1}(\Gamma, \g_{T}\otimes_{T} M)$ under the Dold-Kan correspondence (with $\Ch_{\geq 0}(T)$ replaced by $\Ch^{\leq 0 }(T)$).  Hence $\pi_{i} s\cD_{\rho_{T}}(T\oplus M[j]) \cong H^{1+j-i} (\Gamma, \g_{T}\otimes_{T} M).$
	\end{proof}
	
	We can define the modifying-center version $s\cD_{\rho_{T},Z}$ as in Section \ref{modifying the center}.  Note the fibration sequence (see \cite[(5.7)]{GV18}) $\hofib( \sHom_{\sSet}(X,\BZ(A)) \to \sHom_{\sSet}(X,BZ(T))) \to s\cD_{\rho_{T}}(A) \to s\cD_{\rho_{T},Z}(A)$.  From this, we deduce that $\pi_{i} s\cD_{\rho_{T},Z}(T\oplus M[j]) \cong \pi_{i} s\cD_{\rho_{T}}(T\oplus M[j])$ when $i\neq j+1$, and $\pi_{i+1} s\cD_{\rho_{T},Z}(T\oplus M[i]) = 0$.
	
	For each $v\in S_p$, there is also a modifying-center version $s\cD_{\rho_{T,v},Z}$, resp. $s\cD^{\ord}_{\rho_{T,v},Z}$ of $s\cD_{\rho_{T,v}}$, resp. $s\cD^{\ord}_{\rho_{T,v}}$. Similarly to the global situation, we have:
	\begin{displaymath}
	\pi_{i} s\cD_{\rho_{T,v},Z}(T\oplus M[j]) \cong \left\{\begin{array}{l}{ H^{1+j-i} (\Gamma_v, \g_{T}\otimes_{T} M) \text{ when } i\neq j+1;} \\ { H^{0}(\Gamma_v, \g_{T}\otimes_{T} M)/ (\z_{T}\otimes_{T} M) \text{ when } i=j+1.}\end{array}\right.
	\end{displaymath}
	And 
	\begin{displaymath}
	\pi_{i} s\cD^{\ord}_{\rho_{T,v},Z}(T\oplus M[j]) \cong \left\{\begin{array}{l}{ H^{1+j-i} (\Gamma_v, \b_{T}\otimes_{T} M) \text{ when } i\neq j+1;} \\ { H^{0}(\Gamma_v, \b_{T}\otimes_{T} M)/ (\z_{T}\otimes_{T} M) \text{ when } i=j+1.}\end{array}\right.
	\end{displaymath}
	
	The global nearly ordinary derived deformation functor over $\rho_{T}$ is defined as follows:  $$s\cD^{\ord}_{\rho_{T},Z} = s\cD_{\rho_{T},Z} \times^{h}_{\prod_{v\in S_{p}} s\cD_{\rho_{T,v},Z}} \prod_{v\in S_{p}} s\cD_{\rho_{T,v},Z}^{\ord}.$$  
	Then $\pi_{i} s\cD^{?}_{\rho_{T},Z}(T\oplus M[j])$ ($?=\ord$ or $\emptyset$) depends only on $j-i$.  We denote $\gt_{T,M}^{j-i} s\cD^{?}_{\rho_{T},Z}= \pi_{i} s\cD^{?}_{\rho_{T},Z}(T\oplus M[j])$.
	
	\begin{pro}\label{relordcoh}
		Suppose $(Reg_{v})$ and $(Reg_{v}^{\ast})$.  Let $j \geq i\geq0$ and let $M$ be a finitely generated (classical) $T$-module.  Then $\pi_{i} s\cD^{\ord}_{\rho_{T},Z}(T\oplus M[j]) \cong H^{1+j-i}_{\ord} (\Gamma, \g_{T}\otimes_{T} M).$
	\end{pro}
	
	\begin{proof}
		By preceding discussions, we have the exact sequence
		\begin{align*}
		0 &\to \gt_{T,M}^{-1} s\cD^{\ord}_{\rho_{T},Z} \to \bigoplus_{v \in S_p} H^{0}(\Gamma_v, \b_{T}\otimes_{T} M)/ (\z_{T}\otimes_{T} M) \to \bigoplus_{v \in S_p} H^{0}(\Gamma_v, \g_{T}\otimes_{T} M)/ (\z_{T}\otimes_{T} M) \\ &\to \gt_{T,M}^{0} s\cD^{\ord}_{\rho_{T},Z} \to H^{1}(\Gamma, \g_{T}\otimes_{T} M) \oplus (\bigoplus_{v \in S_p} H^{1}(\Gamma_{v}, \b_{T}\otimes_{T} M)) \to \bigoplus_{v \in S_p} H^{1}(\Gamma_{v}, \g_{T}\otimes_{T} M) \\ &\to \gt_{T,M}^{1} s\cD^{\ord}_{\rho_{T},Z} \to H^{2}(\Gamma, \g_{T}\otimes_{T} M) \to \bigoplus_{v \in S_p} H^{2}(\Gamma_{v}, \b_{T}\otimes_{T} M)  \\ &\to \gt_{T,M}^{2} s\cD^{\ord}_{\rho_{T},Z} \to 0
		\end{align*}
		Note that we have used $H^{2} (\Gamma_{v}, \b_{T}\otimes_{T} M) = 0$ for $v\in S_{p}$.  To see this,  it suffices to show $H^{2} (\Gamma_{v}, \b_{k}) = 0$ by Artinian induction.   By local Tate duality, it suffices to prove 
		$H^{0} (\Gamma_{v}, \b_{k}^{\ast}) = 0$.  But we have a Galois-equivariant ismomorphism 
		$\b_{k}^{\ast} \cong \g_{k}/\n_{k}(1)$, so the result follows from the assumption $(Reg_{v}^{\ast})$.
		
		Under the condition $(Reg_v)$, the map $H^{0}(\Gamma_v, \b_{T}\otimes_{T} M) \to H^{0}(\Gamma_v, \g_{T}\otimes_{T} M)$ is an isomorphism.  Let $L_{v,T,M} = \im(H^1(\Gamma_v,\b_{T}\otimes_{T} M)\to H^1(\Gamma_v,\g_{T}\otimes_{T} M))$, then we have the following exact sequence similar to ($\bigstar$):
		\begin{align*}
		0 \to &H^{0}_{\ord}(\Gamma, \g_{T}\otimes_{T} M) \to H^{0}(\Gamma, \g_{T}\otimes_{T} M) \to 0 \\
		\to &H^{1}_{\ord}(\Gamma, \g_{T}\otimes_{T} M) \to H^{1}(\Gamma, \g_{T}\otimes_{T} M) \to \bigoplus_{v \in S_p} H^{1}(\Gamma_v, \g_{T}\otimes_{T} M)/L_{v,T,M} \\
		\to &H^{2}_{\ord}(\Gamma, \g_{T}\otimes_{T} M) \to H^{2}(\Gamma, \g_{T}\otimes_{T} M) \to \bigoplus_{v \in S_p} H^{1}(\Gamma_v, \g_{T}\otimes_{T} M)\\
		\to &H^{3}_{\ord}(\Gamma, \g_{T}\otimes_{T} M) \to 0
		\end{align*}
		By comparing the two exact sequences above, we get $\gt_{T,M}^{i} s\cD^{\ord}_{\rho_{T},Z} \cong H^{i+1}_{\ord}(\Gamma, \g_{T}\otimes_{T} M)$.
	\end{proof}
	
	Recall that we have a pro-simplicial ring $R^{s,\ord}$ which represents $s\cD^{\ord}_{Z}$.  
	Then $\rho_{T}$ defines a map $$R^{s,\ord} \to \pi_{0}R^{s,\ord} \to T.$$  With this specified map, 
	we regard $R^{s,\ord} \in \mathrm{pro-} \sArtT$, and it's easy to see that $R^{s,\ord}$ 
	represents $s\cD^{\ord}_{\rho_{T},Z}$.  Write $R^{s,\ord} = (R_{k})$ for a projective system $(R_{k})$ in $\sArtT$.  Then 
	\begin{align*}
	\pi_{i} s\cD^{\ord}_{Z} (T \oplus M[j]) &\cong \pi_{i} \varinjlim_{k} \sHom_{\sArtT}(R_k, T \oplus M[j])  \\
	&\cong \pi_{i} \varinjlim_{k} \DK( \tau_{\geq0} [L_{R_k}, M[j]] ) \\
	&\cong H_i \varinjlim_{k} [L_{R_k}, M[j]].
	\end{align*}
	Let's define $[L_{R/\O}, M] = \varinjlim_{k} [L_{R_k /\O}, M]$ for $R= (R_{k}) \in \mathrm{pro-} \sArtT$.  Then $[L_{R^{s,\ord}/\O}, M]$, 
	when regarded as a cochain complex, has the same cohomology groups as the complex $\tau^{\geq0} C^{\bullet+1}_{\ord}(\Gamma,\g_{T}\otimes_{T} M )$. 
	We thus obtain the following corollary:
	
	\begin{cor}\label{ordcoh}
		For every finite $T$-module $M$, there is a quasi-isomorphism $$[L_{R^{s,\ord}/\O}, M]\simeq \tau^{\geq0} C^{\bullet+1}_{\ord}(\Gamma,\g_{T}\otimes_{T} M ).$$
	\end{cor}
	
\vskip 3mm
	
	{\bf Comments:} Let $\rho_T\colon \Gamma\to G(T)$ be an ordinary representation of weight $\mu$, which satisfies $(Reg_v)$ for all $v\in S_p$.
	This means that the cocharacter given by ${\rho}_T\vert_{I_v}\colon I_v\to B(T)/N(T)$ is given (via Artin reciprocity)
	by $\mu\circ rec_v^{-1}\colon I_v\to \cO_v^\times\to \Theta(T)$ (here $\Theta=B/N$ is the standard maximal split torus of $B$).
	In this whole section, if $\rho_T$ is ordinary of weight $\mu$, we could consider instead of the functor $s\cD^{\ord}_{\rho_T}$ 
the subfunctor $s\cD^{\ord,\mu}_{\rho_T}$ of ordinary deformations of 
fixed weight $\mu$. This means we impose as local condition at $v\in S_p$ that 
$$s\cD^{\ord,\mu}_{\rho_{T,v}}(A)=\hofib_{\mu\circ rec_v^{-1}}\left(s\cD^{\ord}_{\rho_{T,v}}(A)\to\sHom(BI_v,B\Theta(A))\right).$$
 Then, $s\cD^{\ord,\mu}_{\rho_T}$ is prorepresentable by a simplical pro-artinian ring $R^{s,\ord}_\mu$ and 
we have an analogue of \propref{relordcoh}:

\begin{pro} \label{relordmucoh} Suppose $(Reg_{v})$ and $(Reg_{v}^{\ast})$.  Let $j \geq i\geq0$ and let $M$ be a finitely generated (classical) $T$-module.  Then $\pi_{i} s\cD^{\ord,\mu}_{\rho_{T},Z}(T\oplus M[j]) \cong H^{1+j-i}_{\ord,str} (\Gamma, \g_{T}\otimes_{T} M).$
\end{pro}

Here $H^{1+j-i}_{\ord,str} (\Gamma, \g_{T}\otimes_{T} M)$ is the cohomology of the subcomplex $C_{\ord,\mu}^\bullet(\Gamma, \g_{T}\otimes_{T} M)$
defined as in Section 4.4.1, replacing $(L_v,\widetilde{L}_v)$ by $(L^\prime_v,\widetilde{L}^\prime_v)$ where $L^\prime_v$ is the image in $H^1(\Gamma_v,\g_T\otimes M)$
of the kernel of $H^1(\Gamma_v,\b_T\otimes M)\to H^1(I_v,(\b_T/\n_T)\otimes M)$, and $\widetilde{L}^\prime_v$ is the inverse image of $L^\prime_v$ in $Z^1(\Gamma_v,\g_T\otimes M)$.

The proof is identical to \propref{relordcoh}. As a corollary we get

	\begin{cor}\label{ordmucoh}
		For every finite $T$-module $M$, there is a quasi-isomorphism $$[L_{R^{s,\ord}_\mu/\O}, M]\simeq \tau^{\geq0} C^{\bullet+1}_{\ord,\mu}(\Gamma,\g_{T}\otimes_{T} M ).$$
	\end{cor}

 In the next section, we shall use these objects with a fixed weight $\mu$.

		\section{Application to the Galatius-Venkatesh homomorphism}\label{GVmorphism}
		
		Let $\Gamma=\Gal(F_S/F)$ for $S=S_p\cup S_\infty$.
		Let $\bar{\rho}\colon \Gamma\to G(k)$ be an ordinary representation of weight $\mu$, which satisfies $(Reg_v)$ for all $v\in S_p$.
		Let $T$ be a finite local $\cO$-algebra and $\rho_T\colon\Gamma\to G(T)$ be an ordinary lifting of weight $\mu$ of $\bar{\rho}$. 
		Let $M$ be a $T$-module which is of $\cO$-cofinite type, that is, whose Pontryagin dual $\Hom_{\cO}(M,K/\cO)$ is finitely generated over $\cO$.
		We use the notations of \defref{dual}. 
		Recall that if $\bar{\rho}\colon \Gamma\to G(k)$ is ordinary automorphic, it is proven under certain assumptions 
		(see \cite[Th.5.11]{CaGe18} and \cite[Lemma 11]{TU20}) that $H^1_{\ord}(\Gamma, \g_T\otimes_T M)$ 
		is finite and $H^1_{\ord^\perp}(\Gamma, \g^\ast_T\otimes_T M)^\vee$ is of $\cO$-cofinite type.
		Let $T_n=T/(\varpi^n)$; it is a finite algebra over $\cO_n=\cO/(\varpi^n)$. 
		Let $\cR=R^{s,\ord}_\mu$, which prorepresents simplicial ordinary deformations of weight $\mu$.
		We consider the simplicial ring homomorphism 
		$$\phi_n\colon \cR\to T_n$$
		given by the universal property for the deformation $\rho_n=\rho_\m \pmod{(\varpi^n)}$.
		Let $T_n=T/(\varpi^n)$; it is a finite algebra over $\cO_n=\cO/(\varpi^n)$. 
		We consider the simplicial ring homomorphism 
		$$\phi_n\colon \cR\to T_n$$
		given by the universal property for the deformation $\rho_n=\rho_\m \pmod{(\varpi^n)}$.
		Let $M_n$ be a finite $T_n$-module.
		Consider the simplicial ring $\Theta_n=T_n\oplus M_n[1]$ concentrated in degrees $0$ and $1$ up to homotopy.
		It is endowed with a simplicial ring homomorphism $\pr_n\colon \Theta_n\to T_n$ given by the first projection.
		Let $L_n(\cR)$ be the set of homotopy equivalence classes of simplicial ring homomorphisms $\Phi\colon \cR\to \Theta_n$ such that
		$\pr_n\circ\Phi=\phi_n$. By \propref{relordmucoh}, there is a canonical bijection 
		$$L_n(\cR) \cong H^2_{\ord,str}(\Gamma,\g_{T_n}\otimes_{T_n}M_n).$$
		
		Moreover, as noticed in \cite[Lemma 15.1]{GV18}, there is a natural map
		$$\pi(n,\cR)\colon L_n(\cR)\to \Hom_{T}(\pi_1(\cR),M_n)$$
		defined as follows. Let $[\Phi]$ be the homotopy class of $\Phi\in \Hom_T(\cR, \Theta_n)$; then $\pi(\cR)(\Phi)$ is the homomorphism 
		which sends the homotopy class $[\gamma]$ of a loop $\gamma$ to 
		$\Phi\circ\gamma\in \Hom_{\sSet}(\Delta[1],M_n[1])=M_n$. Recall a loop $\gamma$ is a morphism of $\sSet$
		$$\gamma\colon \Delta[1]\to \Theta_n$$
		from the simplicial interval $\Delta[1]$ to the simplicial set $\Theta_n$  
		which sends the boundary $\partial \Delta[1]$ to $0$.
		For $G=\GL_N$ and $F$ a CM field (assuming Calegari-Geraghty assumptions), it is proven in \cite{TU20} that
		
		\begin{pro} For any $n\geq 1$, the map $\pi(n,\cR)$ is surjective.
		\end{pro}

		Then, we choose $M_n=\Hom(T,\varpi^{-n}\cO/\cO)$; we take the Pontryagin dual $\pi(n,\cR)^\vee$ and apply Poitou-Tate duality
		
		$$H^2_{\ord,str}(\Gamma,\g_{T_n}\otimes_{T_n}M_n) \cong H^1_{\ord,str}(\Gamma,(\g_{T_n}\otimes_{T_n}M_n)^\ast).
		$$
		We obtain a $T$-linear homomorphism called the mod. $\varpi^n$ Galatius-Venkatesh homomorphism:
		
		$$GV_n\colon \Hom_T(\pi_1(\cR),M_n)^\vee\hookrightarrow H^1_{\ord,str}(\Gamma,(\g_{T_n}\otimes_{T_n}M_n)^\ast)
		$$
		
		The left hand side is $\pi_1(\cR)\otimes\varpi^{-n}/ \cO$  and the right hand side is $\Sel_{\ord,str}(\Ad(\rho_n)(1))$.
		Taking inductive limit on both sides we obtain 
		
		\begin{pro} There is a canonical $T$-linear injection
			$$GV_T\colon \pi_1(R^{s,\ord})\otimes_\cO K/\cO\hookrightarrow \Sel(\Ad(\rho_T)^\vee(1))$$
		\end{pro}
		
		For $G=\GL_N$, $F$ CM and under Calegary-Geraghty assumptions, and for $T$ the non Eisenstein localization of the Hecke algebra acting faithfully on the Betti cohomology,
		it follows from \cite[Theorem 5.11]{CaGe18} that the left-hand side is $\varpi$-divisible of corank $rk(T)$
		and it is proven in \cite[Lemma 11]{TU20} that the right-hand side has corank $rk(T)$.
		For any $\cO$-finitely generated ordinary $\Gamma$-module $M$ such that the Selmer group $H^1_{\ord,str}(\Gamma,M\otimes\Q/\Z)$ is $\cO$-cofinitely generated, 
	we define its Tate-Shafarevitch module as
	$$\Sha(M)=H^1_{\ord,str}(\Gamma,M\otimes\Q/\Z)/H^1_{\ord,str}(\Gamma,M\otimes\Q/\Z)_{\varpi-div}.$$
It is the torsion quotient of $\H^1_{\ord,str}(\Gamma,M\otimes\Q/\Z)$. For any $\cO$-algebra homomorphism
$\lambda\colon T\to \cO$, let $\rho_\lambda=\rho_T\otimes_{\lambda}\cO$.
For $M=\Ad(\rho_\lambda)^\vee(1))$, one shows in \cite[Lemma 11]{TU20}, using Poitou-Tate duality, 
that $\Sha(\Ad(\rho_\lambda)^\vee(1)))$ is Pontryagin dual to $\Sel_{\ord,str}(\Ad(\rho_\lambda))$.

It follows from \cite[Lemma 11]{TU20} that the cokernel of $GV_\lambda$ can be identified to the 
Tate-Shafarevitch group $\Sha(\Ad(\rho_\lambda)^\vee(1))$ in the sense of Bloch-Kato. So that
$$\Coker\,GV_\lambda\cong  \Sel_{\ord,str}(\Ad(\rho_\lambda))^\vee.$$

	{}
	
	Yichang Cai, Jacques Tilouine, LAGA, UMR 7539, Institut Galil\'ee, Universit\'e de Paris 13, USPN, 99 av. J.-B. Cl\'ement, 93430, Villetaneuse, FRANCE
	
\end{document}